\documentclass[10pt]{article}
\usepackage[english]{babel}
\usepackage{url}
\usepackage{inputenc}
\usepackage{amsfonts}
\usepackage{amsmath}
\usepackage{empheq}
\usepackage{graphicx}
\graphicspath{{images/}}
\usepackage{parskip}
\usepackage{fancyhdr}
\usepackage{bbold}
\usepackage{amssymb}
\usepackage{vmargin}
\usepackage{array}
\usepackage{amsthm}
\usepackage[ruled,vlined]{algorithm2e}
\usepackage{hyperref}
\usepackage{caption}
\usepackage{float}
\usepackage{fourier-orns}
\usepackage{listings}
\usepackage{epsfig}
\usepackage{amsmath}               
\newtheorem{assumption}{Assumption}
\newtheorem{theorem}{Theorem}[section]

\newtheorem{lemma}[theorem]{Lemma}
\newtheorem{remark}{Remark}[section]
\theoremstyle{definition}
\newtheorem{definition}{Definition}[section]

\newtheorem{prop}{Proposition}[section]
\usepackage{listings}

\newcommand{\R}{\mathbb{R}}
\newcommand{\om}{{\Omega_d}}
\newcommand{\omq}{{\Omega_q}}

\newcommand{\B}{\mathbb{B}}

\newcommand{\Uoq}{U^{(q)}_0}

\newcommand{\tmu}{\tilde{\mathbf{u}}}
\newcommand{\tu}{\tilde{{u}}}
\newcommand{\cha}{\mathbb{1}_{\Omega_d}}
\newcommand{\chaq}{\mathbb{1}_{\Omega_q}}
\newcommand{\out}{\mathbb{1}_{\mathbb{R}^m\setminus\Omega_d}}
\newcommand{\gdag}{{\gamma^\dagger}}

\newcommand{\bydef}{\stackrel{\mbox{\tiny\textnormal{\raisebox{0ex}[0ex][0ex]{def}}}}{=}}

\usepackage[shortlabels]{enumitem}
\mathtoolsset{showonlyrefs=true}

\usepackage[dvipsnames]{xcolor}

\title{Stability analysis for localized solutions in PDEs and nonlocal equations on $\R^m$}

\author{
Matthieu Cadiot
\footnote{McGill University, Department of Mathematics and Statistics, 805 Sherbrooke Street West, Montreal, QC, H3A 0A9, Canada. {\tt matthieu.cadiot@mail.mcgill.ca}}
}

\date{}

\begin{document}

\maketitle

\begin{abstract}
In this paper, we present a general methodology for investigating the linear stability of localized solutions in PDEs and nonlocal equations on $\R^m$. More specifically, we control the spectrum of the Jacobian $D\mathbb{F}(\tilde{u})$ at a localized solution $\tilde{u}$, enclosing both the eigenvalues and the essential spectrum. Our approach is computer-assisted and is based on a controlled approximation of $D\mathbb{F}(\tilde{u})$ by its Fourier coefficients counterpart on a bounded domain $\om = (-d,d)^m$. We first control the spectrum of the Fourier coefficients operator combining a pseudo-diagonalization and a generalized Gershgorin disk theorem. Then, deriving explicit estimates between the problem on $\om$ and the one on $\R^m$, we construct disks in the complex plane enclosing the eigenvalues of $D\mathbb{F}(\tilde{u})$. Using computer-assisted analysis, the localization of the spectrum is made rigorous and fully explicit. We present applications to the establishment of stability for localized solutions in the planar Swift-Hohenberg PDE, in the planar Gray-Scott model and in the capillary-gravity Whitham equation.
\end{abstract}

\begin{center}
{\bf \small Key words.} 
{ \small Stability analysis, Spectrum of differential operators, Unbounded domains, Gershgorin theorem, Fourier analysis}
\end{center}

\begin{center}
{\bf \small Mathematics Subject Classification (2020)}  \\ \vspace{.05cm}
{\small 	35B35 $\cdot$ 47F05 $\cdot$ 35P05 $\cdot$ 35P15 $\cdot$ 35K58} 
\end{center}


\section{Introduction}\label{sec : introduction}

In this paper, we present a general computer-assisted methodology for the study of linear stability  in partial differential equations (PDEs) and nonlocal equations on $\R^m$. More specifically, we consider a class of autonomous equations given as  
\begin{equation}\label{eq : equation intro}
    \mathbb{F}(u) = 0, ~~ \text{ where } ~  \mathbb{F}(u) = \mathbb{L} u + \mathbb{G}(u),
\end{equation}
and $u : \R^m \to \R$ satisfies $u(x) \to 0$ as $|x|_2 \to \infty$. $\mathbb{L}$ is the linear part of $\mathbb{F}$ and $\mathbb{G}$ is the nonlinear counterpart satisfying $\mathbb{G}(0) = 0$ and $D\mathbb{G}(0)=0$. In particular, we assume that $\mathbb{L}$ is a Fourier multiplier operator, that is it is given by its symbol $l : \R^m \to \mathbb{C}$ as 
\begin{align}\label{eq : definition of l}
    \mathcal{F}\left(\mathbb{L}u\right)(\xi) = l(\xi)  \mathcal{F}({u})(\xi) 
\end{align}
for all $\xi \in \R^m$, where $\mathcal{F}$ is the Fourier transform defined in Section \ref{ssec : notation and main assumption}. If $l$ is polynomial, then $\mathbb{L}$ is a linear differential operator with constant coefficients. Along this paper, we will denote by \emph{localized solutions} the solutions to \eqref{eq : equation intro} vanishing at infinity. Given a smooth zero $\tilde{u} : \R^m \to \R$ of $\mathbb{F}$, we are interested in controlling the spectrum of $D\mathbb{F}(\tilde{u})$, denoted as $\sigma\left(D\mathbb{F}(\tilde{u})\right)$. 

Under  Assumptions \ref{ass:A(1)} and \ref{ass : LinvG in L1} given below, one can decompose the spectrum as $\sigma\left(D\mathbb{F}(\tilde{u})\right) = \sigma_{ess}\left(D\mathbb{F}(\tilde{u})\right) \cup \text{Eig}\left(D\mathbb{F}(\tilde{u})\right)$, where $\sigma_{ess}\left(D\mathbb{F}(\tilde{u})\right)$ amounts for the essential spectrum and $\text{Eig}\left(D\mathbb{F}(\tilde{u})\right)$ the eigenvalues. In particular, under the aforementioned assumptions, we classically prove that 
\begin{equation}\label{eq : essential spectrum}
    \sigma_{ess}\left(D\mathbb{F}(\tilde{u})\right) = \{l(\xi), ~ \xi \in \R^m\},
\end{equation}
that is the essential spectrum is given by the range of the symbol $l$. In this context, the essential spectrum is easily determined and the stability analysis boils down to controlling   $\text{Eig}\left(D\mathbb{F}(\tilde{u})\right)$.

 Such a problematic is of course at the core of linear stability investigations. For instance, it naturally arises in parabolic PDEs of the form
\begin{equation}\label{eq : parabolic intro}
    \partial_t u = \mathbb{L}u + \mathbb{G}(u),
\end{equation}
where, in this case, $\tilde{u}$ is a localized stationary solution to \eqref{eq : parabolic intro}. One other major application of interest is the linear stability (sometimes called spectral stability in this context) of localized traveling waves, also denoted \emph{solitary waves}. Linear stability is essential in understanding the local dynamics surrounding localized stationary solutions and traveling waves. In particular, this field has been abundantly prolific, leading to the development of numerous analytic and numerical approaches. We refer the interested reader to the following non-exhaustive references, and the references therein \cite{ beck_2020, beck_2018, bramburger_2025, kapitula_2013,  sandstede_2019, sandstede_2002, scheel_2023}. 

In general, accessing the spectrum of $D\mathbb{F}(\tilde{u})$ is an highly non-trivial task, since the solution $\tilde{u}$ might  not be known explicitly. This is the case for instance when the existence of $\tilde{u}$ is established as a minimization process (cf. \cite{albert_concentration_compactness, arnesen_existence_solitary,  lions_1984, lions_1984_2} for instance). In this context, the field of computer-assisted proofs (CAPs) becomes a faithful ally. Indeed, CAPs can help establish the existence of $\tilde{u}$ constructively, in the sense that one might prove the existence of $\tilde{u}$ in a vicinity of an approximate solution $u_0$, which is known explicitly (and usually constructed numerically). Then, the use of rigorous numerics might allow to enclose rigorously the spectrum of $D\mathbb{F}(\tilde{u})$, when combined with an adequate analysis. In the next subsection, we present the various developments in CAPs concerning stability problems in differential equations.

\subsection{A computer-assisted approach}

In the last decades, CAPs have become a central tool for the study of PDEs. From the proof of the Feigenbaum conjectures \cite{MR648529}, to the existence of chaos
and global attractor in the Lorenz equations \cite{MR1276767,MR1701385,MR1870856}, the proof of 
 Wright's conjecture \cite{MR3779642},
chaos in the Kuramoto-Sivashinsky PDE \cite{MR4113209}, blowup in 3D
Euler \cite{MR4846737} and imploding solutions for 3D compressible
fluids \cite{MR4862915}, the role of CAPs has become increasingly present in the analysis of ordinary differential equations,
dynamical systems and PDEs.
We refer the interested reader to the following review papers \cite{ gomez_cap,  koch_computer_assisted, nakao_numerical, jb_rigorous_dynmamics} and the book \cite{plum2019numerical} for additional details.

In particular, concerning stability problems, CAPs have known strong advances in multiple directions. As noted above, the strength of a computer-aided approach lies in the fact that a constructive existence proof for $\tilde{u}$ can be established in a vicinity of a numerical approximate solution $u_0$. In practice, one obtains the existence of an explicit $r_0>0$ such that $\|\tilde{u} - u_0\| \leq r_0$, for a well-chosen norm $\|\cdot\|$. Such a result, usually arises from the application of a fixed point argument in a neighborhood of $u_0$ (see \cite{unbounded_domain_cadiot} for instance). This (tight) control on the solution $\tilde{u}$ allows to establish a posteriori stability results on $D\mathbb{F}(\tilde{u})$.\\
~~ In particular, Plum et al.   \cite{Plum2000eigenvalue_domain_decomp,  plum2019numerical, plum1990eigenvalue_homotopy, plum1991bounds, plum_wata2016norm, plum_2009}  have develop a strong framework for the control of the spectrum in PDEs. Their methodology relies on classical Sobolev estimates, combined with various homotopy techniques. In particular, using a domain decomposition homotopy,  this approach allows to treat unbounded domains. This is illustrated in \cite{plum_2004, plum_2009}, in which the  Orr-Sommerfeld equation is treated on the real line. Other applications on the same framework on the real line are given in \cite{breden_2025_bif, breden_2025}. Then, using a finite-elements approach, Liu et al. in \cite{liu_self_adjoint,liu_inverse} provided a general approach for controlling the spectrum in elliptic PDEs on bounded domains. In particular, this approach allows to control the spectral gap of $D\mathbb{F}(\tilde{u})$, an essential step when developing a Banach fixed point approach. Using spectral techniques and Fourier analysis, eigencouples of $D\mathbb{F}(\tilde{u})$  can be established as isolated fixed points of a well-suited zero finding problem, see \cite{breden_2018} for an illustration on a bounded domain and \cite{cadiot2024constructiveproofsexistencestability, unbounded_domain_cadiot} for illustrations of stability of solitary waves on  the real line. \\
Furthermore, the spectrum of $D\mathbb{F}(\tilde{u})$ might also be controlled by using different techniques than the direct computation of the spectrum. Indeed, Beck et al. in \cite{ beck_2020, beck_2018, beck_2022, beck_2024} established a general framework for stability analysis in 1D PDEs thanks to the Maslov index. In particular, such a set-up is compatible with rigorous numerics, as illustrated in \cite{beck_2022} by the rigorous computation of conjugate points  for  one dimensional Sturm-Liouville operators. The number of conjugate points corresponding to the number of unstable direction, one is then able to conclude about linear stability. Finally, Barker et al. in \cite{barker_2024, barker_2016} provided a computer-assisted framework for the construction of an Evans function. Then, the zeros of the Evans function (which can be enclosed via rigorous numerics) provide the localization and multiplicity of the eigenvalues, leading to a spectral analysis.

In the above, we presented various computer-aided approaches allowing the computation of the spectrum of $D\mathbb{F}(\tilde{u})$. One could separate the ``global" methods, as \cite{beck_2022} and \cite{barker_2016}, for which a full control of the unstable directions is obtained at once, to the ``local" methods,   as in \cite{plum_2004} or \cite{cadiot2024constructiveproofsexistencestability, unbounded_domain_cadiot}, in which isolated eigenvalues can be provided. However, the former is, as of now, restricted to one dimensional problems. On the other hand, the latter can be applied to higher dimensional problems, but requires the computation of regions of non-existence of the spectrum in order to conclude about stability.  In other words, one has to provide regions for $\lambda$ in which $D\mathbb{F}(\tilde{u}) - \lambda I$ is invertible. In practice, this can be  numerically very expensive and becomes challenging in higher dimensional problems. Moreover, such a framework is adequate when dealing with simple eigenvalues (which can be established thanks to a fixed point argument), but does not readily apply when dealing with eigenvalues with higher multiplicity. Such hurdles are at the core of this paper, and ask for the development of a Gershgorin type of approach.

\subsection{Localization of the eigenvalues}

For finite dimensional problems, the Gershgorin theorem provides a global enclosure of the spectrum of a matrix using its diagonal entries. Such a result has been used multiple times in CAPs, see \cite{gomez_2021, lessard_2025,  rump_2020, jb_2021} for instance. Generalizations of the Gershgorin theorem to infinite matrices are available, under an additional assumption of compactness (see \cite{FARID19917}). In particular, it was recently used and generalized in the context of CAP in \cite{breden2025turinginstabilitynonlocalheterogeneous} for establishing Turing instability in a reaction-diffusion system. Under very broad assumptions, the authors were able to derive a generalized version of the Gershgorin theorem, applicable whenever a linear operator is expressed on an adequate Schauder basis. In particular, we will show that similar ideas allow to  control of the spectrum of the Jacobian of a PDE with periodic boundary conditions. This is the starting point of our approach, which we expose in Section \ref{sec : control spectrum DF fourier}. 

More precisely, given $q >0$, \eqref{eq : equation intro} has a Fourier coefficients counterpart when expressed on the hypercube $\omq \bydef (-q,q)^m$ with periodic boundary conditions. In fact, the operator $\mathbb{L}$ becomes an infinite diagonal matrix $L_q$ with entries $\left(l(\frac{1}{2q}n)\right)_{n \in \mathbb{Z}^m}$ on the diagonal. Concerning the nonlinear part $\mathbb{G}(u)$, it becomes a nonlinear operator $G_q(U)$, where products $uv$ are turned into discrete convolutions $U*V$ given as 
\[
(U*V)_n = \sum_{k \in \mathbb{Z}^m} U_k V_{n-k}.
\]
Consequently, we obtain a periodic boundary value problem counterpart given as 
\begin{equation}\label{eq : periodic equation intro}
    F_q(U) = 0, ~~ \text{ where } F_q(U) = L_qU + G_q(U)
\end{equation}
and $U$ is a sequence of Fourier coefficients, indexed on $\mathbb{Z}^m$. More details on such Fourier coefficients operators are given in Section \ref{ssec : fourier coefficients counterpart}. As illustrated in \cite{unbounded_domain_cadiot, sh_cadiot, cadiot20242dgrayscottequationsconstructive, cadiot2024constructiveproofsexistencestability}, we expect that localized solutions to \eqref{eq : equation intro} can be approximated thanks to solutions to \eqref{eq : periodic equation intro} if $q$ is big enough. Under such a perspective, we fix $d>0$ and, given an approximate solution $U_0$ to \eqref{eq : periodic equation intro} with $q=d$, we first want to control the spectrum of $DF_d(U_0)$. Under some assumptions on $\mathbb{L}$ and $\mathbb{G}$ given below (cf. Assumptions \ref{ass:A(1)} and \ref{ass : LinvG in L1}), we address this problem in Section \ref{sec : control spectrum DF fourier} thanks to a pseudo-diagonalization of $DF_d(U_0)$ and a generalized Gershgorin theorem.

 At this point, our objective becomes to relate the spectrum of $DF_d(U_0)$ to the one of $D\mathbb{F}(\tilde{u})$, where $U_0$ and $d$ are chosen appropriately. We tackle this problem using the framework developed in \cite{unbounded_domain_cadiot} for PDEs on $\R^m$, which  has recently been  extended  to nonlocal equations in \cite{cadiot2024constructiveproofsexistencestability} and  to systems of PDEs in \cite{cadiot20242dgrayscottequationsconstructive}. In particular, given an hypercube $\om= (-d,d)^m$ with $d$ big enough, we construct an approximate solution $u_0$  thanks to a Fourier series on $\om$ given as 
\begin{equation}\label{eq : approximate solution intro}
    u_0(x) = \mathbb{1}_{\om}(x)\sum_{n \in \mathbb{Z}^m} (U_0)_n e^{\frac{i\pi}{d}n\cdot x}
\end{equation}
for all $x \in \mathbb{R}^m$, where $\mathbb{1}_{\om}$ is the characteristic function on $\om.$ In particular, $U_0 = (U_0)_n$ is a finite dimensional sequence, obtained numerically. Using \cite{unbounded_domain_cadiot}, we are able to prove, after some computer-assisted analysis and the rigorous utilization of a fixed-point argument, that there exists $r_0>0$ and a true solution $\tilde{u} \in \mathcal{H}$ to \eqref{eq : equation intro} such that 
\begin{equation}
    \|\tilde{u} - u_0\|_{\mathcal{H}} \leq r_0,
\end{equation}
where the Hilbert space $\mathcal{H}$ and its norm $\|\cdot\|_{\mathcal{H}}$ are given in \eqref{eq : hilbert space H}. In particular, using the analysis of \cite{unbounded_domain_cadiot}, $r_0$ is known explicitly. Summarizing, given $U_0$ chosen numerically, we are able to construct an approximate solution $u_0$ to \eqref{eq : equation intro} and of the form \eqref{eq : approximate solution intro}. Then, under some explicit bound computations given in \cite{unbounded_domain_cadiot}, we are able to prove the existence of a true solution $\tilde{u}$ in a vicinity of $u_0$, controlled by the radius $r_0$. In particular, the analysis derived in \cite{unbounded_domain_cadiot}, allows to bridge the periodic boundary value problem \eqref{eq : periodic equation intro} to the problem on $\R^m$ \eqref{eq : equation intro}. 

In the present work, we leverage such an analysis and provide a localization of the eigenvalues of $D\mathbb{F}(\tilde{u})$ away from the essential spectrum. More specifically, given $\delta >0$,  we define $\sigma_\delta$ as
\begin{align}\label{def : sigma delta intro}
    \sigma_\delta \bydef \left\{\lambda \in \mathbb{C},~ |l(\xi) - \lambda| > \delta ~\text{ for all } \xi \in \R^m\right\}.
\end{align}
Then $\sigma_\delta \cap \text{Eig}(D\mathbb{F}(\tilde{u}))$ corresponds to the eigenvalues of $D\mathbb{F}(\tilde{u})$ which are $\delta$ away from the essential spectrum (cf. \eqref{eq : essential spectrum}). In particular, given $\delta>0$, we want to localize $\sigma_\delta \cap \text{Eig}(D\mathbb{F}(\tilde{u}))$ using $\sigma_\delta \cap \text{Eig}(DF_d(U_0))$. For this purpose, as hinted in Section 4 of \cite{kato2013perturbation}, we consider a closed Jordan domain $\mathcal{J} \subset \overline{\sigma_\delta}$, that is the closed interior of a Jordan curve. 
Using a first homotopy with the parameter $q$ in \eqref{eq : periodic equation intro} and a second one between $u_0$ and $\tilde{u}$, we obtain the following result, given here in a simplified statement (see Theorem \ref{th : gershgorin unbounded} for the complete result).
\begin{theorem}\label{th : intro}
Denote $(\lambda_n)_{n \in \mathbb{Z}^m}$ the eigenvalues of $DF_d(U_0)$.
   Let $\delta>0$ and let $\mathcal{J}$ be a closed Jordan domain such that $\mathcal{J} \subset \overline{\sigma_\delta}$.  Then  there exists a positive sequence $(\epsilon_n)_{n \in \mathbb{Z}^m}$, depending only on  $\mathcal{J}$, $U_0$ and $d$ such that we have the following :
    Let  $k \in \mathbb{N}$ and $I \subset \mathbb{Z}^m$ such that  $|I| =k$. If $\cup_{n \in {I}} B_{\epsilon_n}(\lambda_n) \subset \mathcal{J}$ and $\left(\cup_{n \in {I}} \overline{B_{\epsilon_n}(\lambda_n)}\right) \bigcap\left( \cup_{n \in \mathbb{Z}^m \setminus {I}} \overline{B_{\epsilon_n}(\lambda_n)}\right) = \varnothing$, then there are exactly $k$ eigenvalues of $D\mathbb{F}(\tilde{u})$ in $\cup_{n \in {I}} \overline{B_{\epsilon_n}(\lambda_n)} \subset \mathcal{J}$ counted with multiplicity.
\end{theorem}
First, note that in the complete statement of Theorem \ref{th : gershgorin unbounded}, we use the pseudo-diagonalization of $DF_d(U_0)$ instead of $DF_d(U_0)$. This has the advantage of providing an explicit expression for the values of  $(\epsilon_n)_{n \in \mathbb{Z}^m}$. Then, it is important to notice that Theorem \ref{th : intro} provides a complete localization of the eigenvalues in $\mathcal{J}$, in the sense that it also provides a control on the multiplicity of eigenvalues. The latter is of major importance when studying stability of localized solutions, since the kernel of $D\mathbb{F}(\tilde{u})$ can have a multiplicity higher than $1$ due to the natural translation and rotation invariances (cf. Section \ref{sec : application proof stability}).

In fact, we illustrate our methodology on highly non-trivial examples, which represent novel results in the stability analysis of localized solutions. First, we investigate stability of localized stationary solutions in the planar Swift-Hohenberg PDE (cf. Section \ref{ssec : application SH}). By controlling the dimension of the kernel, we are able to conclude about nonlinear stability ``modulo trivial invariances" for a first localized pattern (cf. Section \ref{ssec : stable hexagonal pattern}). On the other hand, when studying unstability, Theorem \ref{th : intro} might provide the dimension of the unstable directions. In fact, having access to the dimension of the unstable directions and their relative amplitude can be useful for further analysis of the  local dynamics. We illustrate such an idea on a second localized solution for the Swift-Hohenberg PDE in Section \ref{ssec : unstable square pattern}, and prove that it possesses exactly 3 unstable directions. \\
Then, we demonstrate that our methodology applies to nonlocal equations, and we present an application to a localized traveling wave in the capillary-gravity Whitham equation (see Section \ref{ssec : application whitham}). In such a case, linear stability can be obtained proving that $D\mathbb{F}(\tilde{u})$ has exactly one positive eigenvalue, one simple eigenvalue at $0$ and the rest of the spectrum negative (see \cite{cadiot2024constructiveproofsexistencestability}). This application especially underlies  the ability to localize the eigenvalues and their multiplicity.\\
Finally, we show that our set-up extends to systems of PDEs and we treat the planar Gray-Scott model in Section \ref{sec : application GS}. In fact, we illustrate that our detailed control of the spectrum provides a deeper characterization of the solution $\tilde{u}$. In particular, in the case of the Gray-Scott model, we prove that our solution $\tilde{u}$ is nonlinearly unstable and that it is radially symmetric. The latter is obtained  thanks to a precise control of the kernel of $D\mathbb{F}(\tilde{u})$.

This paper is organized as follows. In Section \ref{sec : presentation of the problem} we introduce the required notation for our analysis, as well as assumptions determining the class of equations \eqref{eq : equation intro} under study. In Section \ref{sec : control spectrum DF fourier}, we present a methodology for controlling the spectrum of the Fourier coefficients operator $DF_d(U_0)$ using a generalized Gershgorin theorem. This approach is then used in Section \ref{sec : control spectrum DF unbounded} in order to localize a subset of the eigenvalues of $D\mathbb{F}(\tilde{u})$ and provide a proof for Theorem \ref{th : intro}. Finally, we present in Section \ref{sec : application proof stability} applications of our methodology to the investigation of linear stability of localized solutions. More specifically, we establish stability results in the planar Swift-Hohenberg PDE, in the capillary-gravity Whitham equation and in the planar Gray-Scott model. All computer-assisted proofs, including the requisite codes, are accessible on GitHub at \cite{julia_cadiot}.

\section{Presentation of the problem}\label{sec : presentation of the problem}

Before presenting our main results and their demonstrations, we provide additional notation and assumptions, completing the description of the Introduction \ref{sec : introduction}.

\subsection{Notations and main assumptions}\label{ssec : notation and main assumption}
 
Let $m \in \mathbb{N}$ and $s \in \mathbb{N}$, we denote by $H^s(\R^m)$ the usual Sobolev space on $\R^m$ and $L^2 = L^2(\mathbb{R}^m)$ the usual Lebesgue space on square integrable functions. More generally, $L^p = L^p(\R^m)$ is the usual $p$ Lebesgue space associated to its  norm $\| \cdot \|_{p}$.
Denote by $\mathcal{B}(L^2)$ the space of bounded linear operators on $L^2$ and given $\B \in \mathcal{B}(L^2)$, denote by $\B^*$ the adjoint of $\B$ in $L^2.$ In particular, given $z \in \mathbb{C}$, denote by $z^*$ the complex conjugate of $z$. Moreover, $|\cdot|_p$ denotes the usual $p$-norm on $\mathbb{R}^m$.

Denote by  $\mathcal{F} :L^2 \to L^2$ the \textit{Fourier transform} operator and
write $\mathcal{F}(f) \bydef \hat{f}$, where $\hat{f}$ is defined as $\hat{f}(\xi) \bydef \int_{\R^m}f(x)e^{-i2\pi x\cdot \xi}dx$ for all $\xi \in \R^m$. Similarly, the \textit{inverse Fourier transform} operator is expressed as $\mathcal{F}^{-1}$ and defined as $\mathcal{F}^{-1}(f)(x) \bydef \int_{\R^m}f(\xi)e^{i2\pi x\cdot \xi}d\xi$. In particular, recall the classical Plancherel's identity
\begin{equation}\label{eq : plancherel definition}
    \|f\|_2 = \|\hat{f}\|_2, \quad \text{for all }f \in L^2.
\end{equation}
Finally, for $f_1,f_2 \in L^1$, the continuous convolution of $f_1$ and $f_2$ is represented by $f_1*f_2$.

In this section, we recall some elements of notation from \cite{unbounded_domain_cadiot} as well as the presentation of our framework. In particular, we introduce the required assumptions for the development of our analysis. Extensive details can be found in \cite{unbounded_domain_cadiot}, and applications to such a framework can be found in \cite{cadiot2024constructiveproofsexistencestability,cadiot20242dgrayscottequationsconstructive,sh_cadiot}.

Recall that we are interested in the following zero-finding problem
\begin{equation*}
    \mathbb{F}(u) = 0, ~~ \text{ where } ~  \mathbb{F}(u) = \mathbb{L} u + \mathbb{G}(u),
\end{equation*}
and $u : \R^m \to \R$ satisfies $u(x) \to 0$ as $|x|_2 \to \infty$.  In particular, $\mathbb{L}$ is given by its symbol $l$ as in \eqref{eq : definition of l}. 
Concerning the nonlinear operator $\mathbb{G}$, we assume the following. $\mathbb{G}$ is a purely nonlinear ($\mathbb{G}(0)=0$ and $D\mathbb{G}(0)=0$) polynomial  operator of order $N_{\mathbb{G}} \in \mathbb{N}$  where $N_{\mathbb{G}} \geq 2$, that is we can decompose it as a finite sum 
\begin{equation}\label{def: G and j}
     \mathbb{G}(u) \bydef \displaystyle\sum_{j = 2}^{N_{\mathbb{G}}}\mathbb{G}_j(u) 
\end{equation}
where $\mathbb{G}_j$ is a sum of monomials of degree $j$ in $u$.
In particular, for $j \in \{2,\dots,N_{\mathbb{G}}\}$, $\mathbb{G}_j$ can be decomposed as follows
\vspace{-.2cm}
\[
\mathbb{G}_j(u) \bydef \sum_{k \in I_j} (\mathbb{G}^1_{j,k}u) \cdots (\mathbb{G}^j_{j,k}u)
\vspace{-.2cm}
\]
where $I_j \subset \mathbb{N}$ is a finite set of indices, and where $\mathbb{G}^p_{j,k}$ is a Fourier multiplier operator for all $1\leq p \leq j$ and $k \in I_j$. Note that under such a formulation, \eqref{eq : equation intro} is an autonomous equation. Now, similarly as in \cite{unbounded_domain_cadiot}, we impose the following assumptions on $\mathbb{L}$ and $\mathbb{G}$.
\begin{assumption}\label{ass:A(1)}
Let $l$ be defined in \eqref{eq : definition of l}. Assume that there exists $\rho >0$ such that $l$ is analytic on the strip $I_\rho \bydef \{z \in \mathbb{C}^m, ~ |\text{Im}(z)|_\infty \leq \rho\}$, where $\text{Im}(z)$ is the imaginary part of $z$.
Moreover, assume that there exists $l_{min} >0$ such that
\begin{equation} \label{eq:l>0}
|l(\xi)| \geq l_{min} \qquad \text{for all } \xi \in \mathbb{R}^m 
\qquad \text{and } \lim_{|\xi|_2 \to +\infty}|l(\xi)| = +\infty. 
\end{equation}
\end{assumption}

\begin{assumption}\label{ass : LinvG in L1}
For all $2 \leq j \leq N_{\mathbb{G}}$, $k \in I_j$ and $1 \leq p \leq j$, define $g_{j,k}^p$ such that
\vspace{-.2cm}
\[
\mathcal{F}\bigg(\mathbb{G}_{j,k}^p u\bigg)(\xi) = g_{j,k}^p(\xi)\hat{u}(\xi),
\quad \text{for all } \xi \in \mathbb{R}^m. 
\vspace{-.2cm}
\]
Moreover, assume that each $\frac{g_{j,k}^p}{l}$ is analytic on $I_\rho$, that $\frac{g_{j,k}^p(\cdot + is)}{l(\cdot + is)} \in L^1$ and that $\displaystyle\lim_{|x|_2 \to \infty}\frac{g_{j,k}^p(x + is)}{l(x + is)} = 0 $ for all $s \in I_\rho$, where $I_\rho$ is defined in Assumption \ref{ass:A(1)}.
\end{assumption}

\begin{remark}
    Note that under Assumptions \ref{ass:A(1)} and \ref{ass : LinvG in L1},  Cauchy's integral theorem is applicable to $\frac{g_{j,k}^p}{l}$ on $I_\rho$ and one can prove that $\mathcal{F}^{-1}\left(\frac{g_{i,k}^p}{l}\right)(x) = \mathcal{O}(e^{-2\pi \rho |x|_1})$. This criterion is one of the main feature of the framework established in \cite{unbounded_domain_cadiot}. In practice, one might be able to show that  $\mathcal{F}^{-1}\left(\frac{g_{i,k}^p}{l}\right)(x) = \mathcal{O}(e^{-2\pi \rho |x|_1})$ can be satisfied for weaker conditions on the integrability of $\frac{g_{i,k}^p}{l}$ (cf. \cite{cadiot2024constructiveproofsexistencestability} or \cite{cadiot20242dgrayscottequationsconstructive} for instance). However, the analytic regularity imposed in Assumptions \ref{ass:A(1)} and \ref{ass : LinvG in L1} is essential for the desired exponential decay.
\end{remark}

\begin{remark}
    If $l$ and each $g^{p}_{i,k}$ are polynomials, then $\mathbb{F}(u) =0$ is an autonomous semi-linear PDE on $\R^m$. In particular $\mathbb{L}$ is a linear differential operator with constant coefficients and $\mathbb{G}(u)$ is polynomial in $u$ and its derivatives.
\end{remark}

As exposed in \cite{unbounded_domain_cadiot}, Assumption \ref{ass:A(1)} allows to define the Hilbert space $\mathcal{H}$ as 
\begin{align}\label{eq : hilbert space H}
    \mathcal{H} \bydef \left\{ u \in L^2, \|u\|_\mathcal{H} \bydef \|\mathbb{L}u\|_2 < \infty \right\}.
\end{align}
In particular $\mathcal{H}$ is associated to a natural inner product $(\cdot,\cdot)_\mathcal{H}$ given by $(u,u)_\mathcal{H} \bydef (\mathbb{L}u,\mathbb{L}u)_2$.
By construction $\mathbb{L} : \mathcal{H} \to L^2$ is an isometric isomorphism. Under Assumption \ref{ass : LinvG in L1}, Lemma 2.4 in \cite{unbounded_domain_cadiot} provides that $\mathbb{G} : \mathcal{H} \to H^1(\R^m)$ is smooth. In fact, the smoothness of $\mathbb{G} : \mathcal{H} \to L^2$ implies that $\mathbb{F} : \mathcal{H} \to L^2$, given in \eqref{eq : equation intro}, is also smooth.  As presented in \cite{unbounded_domain_cadiot}, the operators and spaces introduced in this section have a natural correspondence in the Fourier coefficients world. We recall some of these objects in the following section.

\begin{remark}
    Note that the space $\mathcal{H}$ is denoted $H^l$ in \cite{unbounded_domain_cadiot}, referring to the fact that 
    \[
    \|\mathbb{L}u\|_2 = \|l \hat{u}\|_2  ~~ \text{ for all } u \in \mathcal{H},
    \]
    using Plancherel's identity \eqref{eq : plancherel definition} and the definition of $l$ in \eqref{eq : definition of l}. However, in order to avoid all possible confusions with Sobolev spaces, we choose the more neutral notation $\mathcal{H}$ instead.
\end{remark}

\subsection{Fourier coefficients counterparts}\label{ssec : fourier coefficients counterpart}

Given some $q>0$, we consider the open cube $\Omega_q$ given by 
\begin{equation} \label{eq:definition_Omega_0}
\Omega_q \bydef (-q,q)^m \subset \R^m
\end{equation}
 Similarly to $\mathcal{H}$, we introduce the following Hilbert space $X_q$ on sequences 
\[
     X_q \bydef \{ U=(u_n)_{n\in \mathbb{Z}^m} ~:~ (U,U)_{X_q} < \infty\} \text{ where } (U,V)_{X_q} \bydef \sum\limits_{\underset{}{n \in \mathbb{Z}^m}} u_n v_n \left|l\left(\frac{n}{2q}\right)\right|^{2}.\vspace{-.4cm}
\]
Denote by $\|U\|_{X_q} \bydef \sqrt{(U,U)_{X_q}}$ the induced norm on $X_q$. 
Finally, $\ell^p \bydef \ell^p(\mathbb{Z}^m)$ denotes the usual $p$ Lebesgue space for sequences associated to its natural norm $\| \cdot \|_{p}$.

As in the continuous case, we introduce the notion of a convolution, which in this context is the discrete convolution between sequences $U=(U_n)_n$ and $V = (V_n)_n$, and is given by
\[
    (U*V)_n \bydef \sum\limits_{\underset{}{k \in \mathbb{Z}^m}} U_kV_{n-k},
    \quad \text{for all } n \in \mathbb{Z}^m.\vspace{-.2cm}
\]
The symbol ``$*$'' is used for both discrete and continuous convolutions when no confusion arises. Now, our operators in Section \ref{ssec : notation and main assumption} have a representation in $X_q$ for the  corresponding periodic boundary value problem on $\Omega_q$.   In fact the linear part $\mathbb{L}$ becomes an operator ${L}_q : X_q \to \ell^2$ defined as $L_qU = (l(\frac{n}{2q})u_n)_{n\in \mathbb{Z}^m}$ for all $U \in X.$
Similarly, the nonlinearity $\mathbb{G}$ given in \eqref{def: G and j} has a representation ${G}_q$ in $X_q$ that can be written as 
$
G_q(U) \bydef \sum_{i=1}^{N_\mathbb{G}}{G}_{i,q}(U) $
where ${G}_{i,q}$ is a multi-linear operator on $X_q$ of order $i$, and where products are interpreted as discrete convolutions. Hence, we can define $F_q : X_q \to \ell^2$ and the equivalent of the zero finding problem \eqref{eq : equation intro} on $X_q$ as
\begin{equation}\label{eq : f(u)=0 on X}
  F_q(U) \bydef L_qU + G_q(U), \quad \text{for } U \in X_q.
\end{equation}

In the rest of the paper, we fix some $d>0$ and denote $F = F_d$, $L = L_d$ and $G=G_d$ for simplicity. When necessary, we will use the subscript $d$ to avoid potential confusions.

\subsubsection{Passage between $L^2(\R^m)$ and $\ell^2$, and vice versa}

As explained in the introduction, our goal is to analyse the properties of $D\mathbb{F}(\tilde{u})$ thanks to those of $DF(U_0)$ in Fourier coefficients space, where $U_0 \in X_d$ for some $d>0$. Consequently, we need to introduce some transformations and associated spaces to switch from operators on $L^2(\R^m)$ to operators on $\ell^2$.
Let us recall the maps $\gamma_q : L^2 \to \ell^2$ and $\gamma^\dagger_q : \ell^2 \to L^2$  defined as 
\begin{align}\label{eq : gamma and gamma dagger functions}
    \gamma_q(u)  \bydef  \left(\frac{1}{|\omq|}\int_{\Omega_q} u(x)e^{-2\pi i \tilde{n}\cdot x}dx \right)_{n \in \mathbb{Z}^m} \text{ and }  \left(\gamma_q^\dagger(U)\right)(x) \bydef \chaq(x) \sum_{n \in \mathbb{Z}^m}u_ne^{2\pi i \tilde{n}\cdot x}
\end{align}
for all $u \in L^2$,
for all $U = (u_n)_{n \in \mathbb{Z}^m} \in \ell^2$ and all $x \in \R^m$, where $\chaq$ is the characteristic function on $\omq$. 

 Now, given a Hilbert space $H$ of functions defined on $\mathbb{R}^m$, denote by $H_\omq \subset H$ the subspace of $H$ defined as
\begin{equation}\label{def : Homega}
   H_\omq \bydef \{u \in H ~: ~ \text{supp}(u) \subset \overline{\omq} \}.
\end{equation}
For instance, we have $L^2_\omq \bydef \{u \in L^2 ~:~ \text{supp}(u) \subset \overline{\omq} \}$ or $\mathcal{H}_\omq \bydef \{u \in \mathcal{H} ~:~ \text{supp}(u) \subset \overline{\omq} \}$. Then, recall that $\mathcal{B}(L^2)$ (respectively $\mathcal{B}(\ell^2)$) denotes the space of bounded linear operators on $L^2$ (respectively $\ell^2$) and denote by $\mathcal{B}_\omq(L^2)$ the following subspace of $\mathcal{B}(L^2)$
\begin{equation}\label{def : Bomega}
    \mathcal{B}_\omq(L^2) \bydef \{\mathbb{B}_\omq \in \mathcal{B}(L^2) ~:~  \mathbb{B}_\omq = \chaq \mathbb{B}_\om \chaq\}.
\end{equation}
Finally, define $\Gamma_q : \mathcal{B}(L^2) \to \mathcal{B}(\ell^2)$ and $\Gamma^\dagger_q : \mathcal{B}(\ell^2) \to \mathcal{B}(L^2)$ as follows
\begin{align}\label{def : Gamma and Gamma dagger}
    \Gamma_q(\mathbb{B}) \bydef \gamma_q \mathbb{B} \gdag_q ~~ \text{ and } ~~  \Gamma^\dagger_q(B) \bydef \gamma^\dagger_q {B} \gamma_q. 
\end{align}
Now, we recall Lemma 3.2 from \cite{unbounded_domain_cadiot} which provides some useful properties of the previously defined transformations that will help us pass from $\ell^2$ to $L^2$ and vice versa.
\begin{lemma}\label{lem : gamma and Gamma properties}
    The map $\sqrt{|\omq|} \gamma_q : L^2_\omq \to \ell^2$ (respectively $\Gamma_q : \mathcal{B}_\omq(L^2) \to \mathcal{B}(\ell^2)$) is an isometric isomorphism whose inverse is given by $\frac{1}{\sqrt{|\omq|}} \gdag : \ell^2 \to L^2_\omq$ (respectively $\Gamma^\dagger_q :   \mathcal{B}(\ell^2) \to \mathcal{B}_\omq(L^2)$).\\
    In particular,
    \begin{align}\label{eq : parseval's identity}
        \|u\|_2 = \sqrt{\omq}\|U\|_2 \text{ and } \|\mathbb{B}_\omq\|_2 = \|B\|_2
    \end{align}
    for all $u \in L^2_\omq$ and $\mathbb{B}_\omq \in \mathcal{B}_\omq(L^2)$, and where $U \bydef \gamma_q(u)$ and $B \bydef \Gamma_q(\mathbb{B}_\omq)$.
\end{lemma}
In particular, given a bounded operator $B : \ell^2 \to \ell^2$, then $\Gamma_q^\dagger(B)$ provides an equivalent bounded operator on $L^2$ with $\|\Gamma_q^\dagger(B)\|_2 = \|B\|_2$. This fact will notably be useful to approximate operators on $L^2 \to L^2$  thanks to bounded operators on Fourier coefficients. 

Finally, the computer-assisted part of our approach involves the rigorous computation of vector and matrix quantities on the computer. In particular, it requires considering finite dimensional projections.
Let $N \in \mathbb{N}$ and define the projections $\pi^N : \ell^2 \to \ell^2$ and $\pi_N : \ell^2 \to \ell^2$ as follows
\vspace{-.2cm}
\begin{align}\label{def : projection on size N}
    \left(\pi^N(U)\right)_n  =  \begin{cases}
          U_n,  & n \in I^N \\
              0, &n \notin I^N
    \end{cases} 
     ~~ \text{and} ~~
     \left(\pi_N(U)\right)_n  =  \begin{cases}
          0,  & n \in I^N \\
              U_n, &n \notin I^N
    \end{cases}
    \vspace{-.2cm}
\end{align}
for all $n \in \mathbb{Z}^m$, where $I^N \bydef \{n \in \mathbb{Z}^m,~  |n_1| \leq N,\dots,|n_m|\leq N\}.$ Specifically, the above projection operators will help us separate finite-dimensional computations (which will be handled on the computer) and theoretical estimates.

\subsection{Some spectral properties}

The framework of \cite{unbounded_domain_cadiot} provides a general framework for proving constructively the existence of zeros of $\mathbb{F}$ in $\mathcal{H}$, around an approximate solution $u_0 \in \mathcal{H}$. We recall some of the elements of such a construction.

Using the methodology of \cite{unbounded_domain_cadiot}, $u_0$ is constructed via its Fourier coefficients $U_0$, and more precisely  $u_0 \bydef \gamma_d^\dagger(U_0)$. Moreover, we assume that $U_0 \in \ell^2$ satisfies $U_0 = \pi^N U_0$ (that is $U_0$ only has a finite number of non-zero elements). In practice, $U_0$ is determined numerically and $N$ represents the size of the numerical truncation. By construction, note that supp$(u_0) \subset \overline{\om}$ by definition of $\gamma_d^\dagger.$ In particular, using Section 4.1 in \cite{unbounded_domain_cadiot}, we can ensure that $u_0 \in \mathcal{H}$ using a projection on the set of function with zero trace on $\Omega_d$. Consequently, without loss of generality, we assume for the rest of this manuscript that we have access to an approximate zero of $\mathbb{F}$ such that 
\begin{equation}\label{eq : approximate solution}
    u_0 = \gamma_d^\dagger(U_0) \in \mathcal{H} ~~ \text{ and } ~~ U_0 = \pi^N U_0 \in \ell^2.
\end{equation}
Now, using the set-up exposed in \cite{unbounded_domain_cadiot}, and especially Theorem 4.6, suppose that we were able to prove the existence of $\tilde{u} \in \mathcal{H}$ such that $\mathbb{F}(\tilde{u}) = 0$ and 
\begin{equation}\label{eq : defect with true solution}
    \|\tilde{u} - u_0\|_{\mathcal{H}} \leq r_0
\end{equation}
for some known $r_0$. Then, our strategy is to control the spectrum of $D\mathbb{F}(\tilde{u}) : \mathcal{H} \to L^2$ thanks to the one of $DF(U_0) : {X_d} \to \ell^2.$

Given a bounded linear operator $\mathbb{K}$, denote by $\sigma(\mathbb{K})$ the spectrum, $\text{Eig}(\mathbb{K})$ the eigenvalues and $\sigma_{ess}(\mathbb{K})$ the essential spectrum of $\mathbb{K}$. Moreover, given $\delta \geq 0$, let $\sigma_\delta$ be the following set
\begin{align}\label{def : sigma delta}
    \sigma_\delta \bydef \left\{\lambda \in \mathbb{C},~ |l(\xi) - \lambda| > \delta ~\text{ for all } \xi \in \R^m\right\}.
\end{align}
Then we derive the following lemma which allows to separate the spectrum of $D\mathbb{F}(\tilde{u})$ thanks to $\sigma_0.$
\begin{lemma}\label{lem : separation of the spectrum sigma 0}
    We have the following identities
    \begin{align*}
        \sigma_0 \cap \sigma\left(D\mathbb{F}(\tilde{u})\right) &= \text{Eig}\left(D\mathbb{F}(\tilde{u})\right)\\
        \left(\R^m \setminus \sigma_0\right) \cap \sigma\left(D\mathbb{F}(\tilde{u})\right) &= \sigma_{ess}\left(D\mathbb{F}(\tilde{u})\right).
    \end{align*}
    Moreover, the eigenvalues of $D\mathbb{F}(\tilde{u})$ have a finite multiplicity.
\end{lemma}
\begin{proof}
First, notice that the proof of the second identity  is a direct consequence of Weyl's perturbation theory (cf. \cite{kato2013perturbation} Theorem 5.35 in Chapter IV for instance). Indeed,  using Lemma 3.1 in \cite{unbounded_domain_cadiot}, we known that $D\mathbb{G}(\tilde{u})$ is relatively compact with respect to $\mathbb{L}$. This implies that $\sigma_{ess}\left(D\mathbb{F}(\tilde{u})\right) = \sigma_{ess}\left(\mathbb{L}\right) =\R^m \setminus \sigma_0 = \left(\R^m \setminus \sigma_0\right) \cap \sigma\left(D\mathbb{F}(\tilde{u})\right).$ Moreover, this also yields that $$\text{Eig}\left(D\mathbb{F}(\tilde{u})\right) \subset  \sigma_0 \cap \sigma\left(D\mathbb{F}(\tilde{u})\right).$$
     If $\lambda \in \sigma_0 \cap \sigma\left(D\mathbb{F}(\tilde{u})\right)$, then $\mathbb{L}-\lambda I : \mathcal{H} \to L^2$ is boundedly invertible and, using the proof of Lemma 3.1 in \cite{unbounded_domain_cadiot}, we obtain that $D\mathbb{G}(\tilde{u})$ is relatively compact with respect to $\mathbb{L}-\lambda I$. This implies that $D\mathbb{F}(\tilde{u}) - \lambda I : \mathcal{H} \to L^2$ is a Fredholm operator of index zero and therefore $\lambda \in \text{Eig}\left(D\mathbb{F}(\tilde{u})\right).$   Moreover, the Fredholm alternative provides that $\lambda$ has a finite multiplicity (see \cite{brezis2011functional}).
\end{proof}

In practice, using the above lemma, one can easily determine $\sigma_{ess}\left(D\mathbb{F}(\tilde{u})\right)$ thanks to the range of the symbol $l : \R^m \to \mathbb{C}$, as already exposed in \eqref{eq : essential spectrum}.  Now we derive a result which will be useful when constructing explicit disks for the enclosure of the eigenvalues of $D\mathbb{F}(\tilde{u})$.
\begin{prop}\label{prop : minimum of the spectrum full}
    There exists $t \in \mathbb{C}$ such that both $DF(U_0) + tI : {X_d} \to \ell^2$ and $D\mathbb{F}(\tilde{u}) + tI : \mathcal{H} \to L^2$ are boundedly invertible. 
\end{prop}
\begin{proof}
Using that $l : \R^m \to \mathbb{C}$ is smooth and $|l(\xi)| \to +\infty$ as $|\xi|_2 \to +\infty$,  there exists $t_0 \in \mathbb{C}$ and $S >0$ such that $|l(\xi) + s t_0| \geq \frac{|l(\xi)| + |t_0|s}{2} $ for all $\xi \in \mathbb{R}^m$ and all $s \geq S$. Given $s \geq S$, this implies that $\mathbb{L} + st_0 I : \mathcal{H} \to L^2$ is boundedly invertible. Moreover, we have 
\begin{align}\label{eq : compactness fourier coeffs}
    \|DG(U_0)(L + st_0 I)^{-1}\|_2 &\leq \|DG(U_0)L^{-1}\|_2  \|L (L+st_0 I)^{-1}\|_2 \\
    &\leq \|DG(U_0)L^{-1}\|_2\sup_{\xi \in \R^m}\frac{2|l(\xi)|}{|l(\xi)|+ |t_0|s}.
\end{align}
Since $DG(U_0) : X_d \to \ell^2$ is bounded, we obtain that $\|DG(U_0)L^{-1}\|_2 < \infty$.
Consequently, we obtain that $ \|DG(U_0)(L + st_0 I)^{-1}\|_2 \to 0$ as $s \to \infty$. A similar reasoning holds for $\|D\mathbb{G}(\tilde{u})(\mathbb{L} + st_0 I)^{-1}\|_2$. In particular,  we get  that there exists $s$ sufficiently big such that 
\begin{align*}
    \|DG(U_0)(L + st_0 I)^{-1}\|_2 \leq \frac{1}{2}   ~~ \text{ and } ~~ \lim_{s \to \infty}\|D\mathbb{G}(\tilde{u})(\mathbb{L} + st_0 I)^{-1}\|_2 \leq \frac{1}{2}, 
\end{align*}
which implies that $DF(U_0) + t I =  \left[I +  DG(U_0)(L + t I)^{-1}\right](L + t I) : {X_d} \to \ell^2$ and $D\mathbb{F}(\tilde{u}) = \left[I + D\mathbb{G}(\tilde{u})(\mathbb{L} + t I)^{-1}\right] (\mathbb{L} + t I) : \mathcal{H} \to L^2 $ are boundedly invertible for $t = st_0$, using a Neumann series argument.
\end{proof}
Now, our goal is to control the eigenvalues of $D\mathbb{F}(\tilde{u})$ away from the essential spectrum. In particular, given $\delta >0$, we want to determine the position and the multiplicity of the eigenvalues of  $D\mathbb{F}(\tilde{u})$ in $\sigma_{\delta}.$ More specifically, we want to enclose these eigenvalues thanks those of $DF(U_0)$.

\section{Spectrum of \boldmath$DF(U_0)$ \unboldmath  using Gershgorin disks}\label{sec : control spectrum DF fourier}

In this section, we present a computer-assisted strategy to enclose the spectrum of $DF(U_0)$. Our approach is based on a pseudo-diagonalization of $DF(U_0)$, that is we construct a linear operator $P$ and its inverse, and study the spectrum of $\mathcal{D} = P^{-1}DF(U_0)P$ instead. By construction $\mathcal{D}$ is supposed to be diagonally dominant, which hints to the Gershgorin theorem. In particular, we derive a generalization of the Gershgorin theorem in the specific context of Assumptions \ref{ass:A(1)} and \ref{ass : LinvG in L1}. Note that a similar framework is also described more generally in \cite{breden2025turinginstabilitynonlocalheterogeneous}. First, we prove that the spectrum of $DF(U_0)$ consists of eigenvalues with finite multiplicity.

\begin{lemma}\label{eq : spectrum of DFU0 is point}
    $\sigma\left(DF(U_0)\right)$ consists of eigenvalues with finite multiplicity.
\end{lemma}

\begin{proof}
Let $t$ be as in Proposition \ref{prop : minimum of the spectrum full}. Using the proof of Proposition \ref{prop : minimum of the spectrum full}, we have that $L+tI : X_d \to \ell^2$ is boundedly invertible. Then, using Assumption \ref{ass : LinvG in L1} and the proof of Lemma 2.4 in \cite{unbounded_domain_cadiot}, we obtain that $DG(U_0)(L+t I)^{-1} : \ell^2 \to \ell^2$ is bounded. Consequently, we have that 
\begin{align*}
    (DF(U_0)+tI)^{-1} = (L+tI)^{-1}\left(I + DG(U_0)(L+t)^{-1}\right)^{-1}.
\end{align*}
Now, we obtain that $(L+tI)^{-1} : \ell^2 \to \ell^2$ is compact thanks to Assumption \ref{ass:A(1)}. Then,\\ $\left(I + DG(U_0)(L+t)^{-1}\right)^{-1} : \ell^2 \to \ell^2$ is a bounded linear operator, which implies that $(DF(U_0)+tI)^{-1} : \ell^2 \to \ell^2$ is compact. This concludes the proof.
\end{proof}

Now, we expose the construction of the pseudo-diagonalization of $DF(U_0)$. In particular, we describe the construction of $P$, leading to the establishment of $\mathcal{D}$.

We choose the linear operator $P$ as  $P = P^N + \pi_N : \ell^2 \to \ell^2$ where $P^N = \pi^N P^N \pi^N$ and $P^N : \pi^N \ell^2 \to \pi^N\ell^2$ is invertible. 
In particular, $P$ has a bounded inverse given by $P^{-1} = (P^N)^{-1} + \pi_N$ where $(P^N)^{-1}$ has to be understood as the inverse of $P^N : \pi^N \ell^2 \to \pi^N\ell^2$. In practice we choose $P^N$ such that the columns of $P^N$ approximate the eigenvectors of $\pi^N DF(U_0)\pi^N$. This construction is achieved thanks to rigorous numerics. In fact, rigorous numerics also allow to establish the rigorous inverse of $P^N$.  Then we define 
\begin{equation}\label{eq : pseudo diagonalization DF}
    \mathcal{D} = P^{-1}DF(U_0)P, ~~ S \bydef diag(\mathcal{D}) ~~ \text{ and } ~~ R \bydef \mathcal{D} - S,
\end{equation}
where $diag(\mathcal{D})$ is the diagonal part of the infinite matrix $\mathcal{D}$. Moreover, we define the diagonal entries $(\lambda_n)_{n \in \mathbb{Z}^m}$ of $S$ as 
\begin{equation}\label{def : lambda n}
    (SU)_n = \lambda_n U_n
\end{equation}
for all $n \in \mathbb{Z}^m$ and all $U =(U_n)_{n \in \mathbb{Z}^m} \in \ell^2$.
Our goal is to enclose the spectrum of $DF(U_0)$ thanks to  the $\lambda_n$'s. This is achieved using a generalized Gershgorin theorem from \cite{FARID19917}, which we present in the next lemma under our notations.
\begin{lemma}\label{lem : gershgorin matrix}
Given $n \in \mathbb{Z}^m$, define $r_n \geq 0$ as 
\begin{align*}
    r_n \bydef \sum_{k\in \mathbb{Z}^m} \left|R_{n,k}\right|,
\end{align*}
where $R$ is defined in \eqref{eq : pseudo diagonalization DF}.
 Moreover, let us denote by $B_n \bydef \overline{B_{r_n}(\lambda_n)} \subset \mathbb{C}$ the $n$-th Gershgorin (closed) disc of radius $r_n$ and centered at $\lambda_n$ (defined in \eqref{def : lambda n}). 
 Then the spectrum of $DF(U_0)$ consists of eigenvalues lying in the set $\bigcup_{n \in \mathbb{Z}^m} B_n$. Moreover, any set of $k$ Gershgorin discs whose union is disjoint from all other Gershgorin discs intersects $\sigma\left(DF(U_0)\right)$ in a finite set of eigenvalues of $DF(U_0)$ with total algebraic multiplicity $k.$
\end{lemma}

\begin{proof}
First, using the construction of $P$, we can decompose $\mathcal{D}$ as follows
\begin{align}\label{eq : decomposition D}
    \mathcal{D} = \left(P^{N}\right)^{-1}DF(U_0)P+ \pi_N DF(U_0)P.
\end{align}
Note that $\left(P^{N}\right)^{-1}DF(U_0)P^N$ has a  finite dimensional range and therefore, there exists $s_1 \in \mathbb{C}$ big enough in amplitude such that 
\begin{align*}
    |\lambda_n + s_1| > \frac{1}{2}  \sum_{k\in \mathbb{Z}^m, ~k\neq n} \left|\mathcal{D}_{n,k}\right| = \frac{1}{2}  \sum_{k\in \mathbb{Z}^m} \left|R_{n,k}\right|
\end{align*}
for all $n \in I^N$. Then, since $DG(U_0)L^{-1} : \ell^2 \to \ell^2$ is compact  and $|l(\tilde{n})| \to \infty$ as $|n| \to \infty$, there exists $s_0 \in \mathbb{C}$ sufficiently big in amplitude such that 
\[
|l(\tilde{n}) + s_0| > \frac{1}{2} \sum_{k \in \mathbb{Z}^m, ~k \neq n} \left|\left(DG(U_0)\right)_{n,k}\right|
\]
for all $n \in \mathbb{Z}^m$. This implies that 
\begin{align*}
     |\lambda_n + s_0| > \frac{1}{2}  \sum_{k\in \mathbb{Z}^m} \left|R_{n,k}\right|
\end{align*}
for all $n \in \mathbb{Z}^m\setminus I^N$. Consequently, we can find some $s \in \mathbb{C}$ big enough in amplitude such that 
\begin{align*}
     |\lambda_n + s| >  \frac{1}{2}  \sum_{k\in \mathbb{Z}^m } \left|R_{n,k}\right|
\end{align*}
for all $n \in \mathbb{Z}^m$. Consequently, the matrix $\mathcal{D} + sI$ satisfies the hypotheses (1) and (2) of Theorem 2.1 in \cite{FARID19917}. Now, using that $U_0$ possesses a finite number of non-zero coefficients, we can easily prove that $DG(U_0)L^{-1} : \ell^\infty \to \ell^\infty$ is a bounded linear operator. Consequently, hypothesis (3) is satisfied as well for $p = \infty$ and we conclude the proof using the aforementioned theorem in \cite{FARID19917}.
\end{proof}

Using the above result, we can derive explicit values for the radii $(r_n)$ based on the operator $P$. In particular, we separate final dimensional computations, which will be handled rigorously on the computer, and infinite sums, which will have to be estimated analytically.

\begin{lemma}
Let $n \in \mathbb{Z}^m$ and let $r_n$ by defined as in Lemma \ref{lem : gershgorin matrix}, then
    \begin{align*}
        r_n = \begin{cases}
          \displaystyle \sum_{k \in I^N, k \neq n} \left|\left(\left(P^{N}\right)^{-1}DF(U_0)P^N\right)_{n,k}\right| + \sum_{k \in \mathbb{Z}^m \setminus I^N} \left|\left(\left(P^{N}\right)^{-1}DG(U_0)\right)_{n,k}\right| &\text{ if } n \in I^N\\
           \displaystyle \sum_{k \in I^N} \left|\left(DG(U_0)P^N\right)_{n,k}\right| + \sum_{k \in \mathbb{Z}^m \setminus I^N,  k \neq n} \left|\left(DG(U_0)\right)_{n,k}\right| &\text{ if } n \in \mathbb{Z}^m \setminus I^N 
        \end{cases}
    \end{align*}
\end{lemma}

\begin{proof}
Using the definition of $P$ and $P^{-1}$, we have 
\begin{align*}
    \mathcal{D} =  \left(P^{N}\right)^{-1}DF(U_0)P^N + \left(P^{N}\right)^{-1}DF(U_0)\pi_N + \pi_N DF(U_0)P^N + \pi_N DF(U_0)\pi_N.
\end{align*}
Now, since $L$ is diagonal, we have that $L\pi_N = \pi_N L\pi_N$ and therefore 
\[
\left(P^{N}\right)^{-1}DF(U_0)\pi_N  = \left(P^{N}\right)^{-1}DG(U_0)\pi_N ~~ \text{ and } ~~ \pi_N DF(U_0)P^N = \pi_N DG(U_0)P^N.
\]
This concludes the proof.
\end{proof}

In practice, given $n \in I^N$, then the computation of $r_n$ involves a finite number of vector norms. Indeed, since $G$ is polynomial of order $N_\mathbb{G}$, we have
\[
\sum_{k \in \mathbb{Z}^m \setminus I^N} \left|\left(\left(P^{N}\right)^{-1}DG(U_0)\right)_{n,k}\right| = \sum_{k \in I^{N_\mathbb{G} N} \setminus I^N} \left|\left(\left(P^{N}\right)^{-1}DG(U_0)\right)_{n,k}\right|.
\]
In practice, for each $n \in I^N$, $\lambda_n$ and $r_n$ can be obtained using rigorous numerics on vectors and matrices. For the case $n \in \mathbb{Z}^m \setminus I^N$, notice that 
\[
\lambda_n = \left(\pi_N(L+DG(U_0))\pi_N\right)_{n,n} = l(\tilde{n}) + (DG(U_0))_{n,n}.
\]

Now, suppose that $n \in \mathbb{Z}^m \setminus I^{N_\mathbb{G}N}$, then since ${G}$ is polynomial of order $N_\mathbb{G}$, then 
\[
\sum_{k \in I^N} \left|\left(DG(U_0)P^N\right)_{n,k}\right| = 0. 
\]
This implies that for all $n \in I^{N_\mathbb{G}N} \setminus I^N$, $r_n$ can be computed using a finite number of vector norms. Moreover, if $n \in \mathbb{Z}^m \setminus I^{N_\mathbb{G}N}$, we have
\[
r_n = \sum_{k \in \mathbb{Z}^m \setminus I^N,  k \neq n} \left|\left(DG(U_0)\right)_{n,k}\right|
\]
which can be estimated using Banach algebra properties on sequences (for instance $\|U*V\|_1 \leq  \|U\|_1\|V\|_1$ for all $U,V \in \ell^1$) or Young's inequality for the convolution (for instance $\|U*V\|_2 \leq  \|U\|_2\|V\|_1$ for all $U \in \ell^2$, $V\in \ell^1$). 

Now that we provided a general strategy for the rigorous enclosure of the spectrum of $DF(U_0)$, it remains to relate the eigenvalues of  $DF(U_0)$ with those of $D\mathbb{F}(\tilde{u})$. 
%
%
%
%
%
%
%
\section{Control of the spectrum of \boldmath $D\mathbb{F}(\tilde{u})$ \unboldmath thanks to \boldmath $DF(U_0)$ \unboldmath}\label{sec : control spectrum DF unbounded}

In this section, we demonstrate how a subset of the eigenvalues of $D\mathbb{F}(\tilde{u})$ can be enclosed thanks to the ones of $DF(U_0)$. In particular, we want to use the analysis introduced in \cite{unbounded_domain_cadiot} to create Gershgorin disks containing eigenvalues of $D\mathbb{F}(\tilde{u})$ in a closed Jordan domain $\mathcal{J} \subset \overline{\sigma_\delta}$, for some $\delta>0$. Our strategy is to construct a first homotopy from $\mathcal{D}$ to $D\mathbb{F}(u_0)$ and a second one from $D\mathbb{F}(u_0)$ to $D\mathbb{F}(\tilde{u})$. Then, we enclose the eigenvalues in $\mathcal{J}$ of each operator along the homotopies thanks to disks centered at the diagonal entries of $\mathcal{D}$. 

In the next lemma, we provide a localization of the eigenvalues of $D\mathbb{F}(\tilde{u})$ thanks to the pseudo-diagonalization \eqref{eq : pseudo diagonalization DF}. In particular, we provide explicit disks centered at the diagonal entries of $\mathcal{D}$ containing the eigenvalues of $D\mathbb{F}(\tilde{u})$ in $\mathcal{J}$.
\begin{lemma}\label{lem : link of spectrum unbounded to U0}
  Let $\delta>0$, let $\mathcal{J}$ be a closed Jordan domain in $\overline{\sigma_\delta}$ and let $\lambda \in \mathcal{J}\cap \sigma\left(D\mathbb{F}(\tilde{u})\right)$. 
Furthermore, let $t \in \mathbb{C}$ be given as in Lemma \ref{prop : minimum of the spectrum full} and recall  $R,S$ defined in \eqref{eq : pseudo diagonalization DF}.  Moreover, suppose that $S+tI$ is invertible.
  Now, let us introduce various constants satisfying the following 
  \begin{align*}
     \mathcal{Z}_{u,1} &\geq  \sup_{\mu \in \mathcal{J}} \left\|\out \left(\mathbb{L} - \mu I\right)^{-1} D\mathbb{G}(u_0)\right\|_2\\
\mathcal{Z}_{u,2} & \geq \sup_{\mu \in \mathcal{J}} \| \cha \left(\Gamma^\dagger\left((L- \mu I )^{-1}\right) - (\mathbb{L}- \mu I)^{-1}\right)D\mathbb{G}(u_0) \|_2\\
\mathcal{Z}_{u,3} & \geq \sup_{\mu \in \mathcal{J}}  \|\pi^N (S+ tI)^{-1}P^{-1}(L-\mu I)\|_2 \| \cha \left(\Gamma^\dagger\left((L- \mu I )^{-1}\right) - (\mathbb{L}- \mu I)^{-1}\right)D\mathbb{G}(u_0) \|_2\\
     \mathcal{C}_1   &\geq  \frac{1}{r_0}\sup_{\mu \in \mathcal{J}} \left\| \left(\mathbb{L} - \mu I\right)^{-1} \left(D\mathbb{G}(u_0)-D\mathbb{G}(\tilde{u})\right)\right\|_2\\
     \mathcal{C}_2   &\geq  \frac{1}{r_0}\sup_{\mu \in \mathcal{J}} \|\pi^N (S+ tI)^{-1}P^{-1}(L-\mu I)\|_2 \left\| \cha \left(\mathbb{L} - \mu I\right)^{-1} \left(D\mathbb{G}(u_0)-D\mathbb{G}(\tilde{u})\right)\right\|_2\\
     Z_{1,1} &\geq \sup_{\mu \in \mathcal{J}}\|\pi_N (L -\mu I)^{-1}R \pi^N\|_2, ~~~~ Z_{1,2} \geq \sup_{\mu \in \mathcal{J}}\|\pi_N (L -\mu I)^{-1}R\pi_N\|_2\\
       Z_{1,3} &\geq\| \pi^N(S+ tI)^{-1}R\pi^N\|_2, ~~~~ Z_{1,4} \geq \| \pi^N(S+ tI)^{-1}R\pi_N\|_2.
  \end{align*}
  If 
  \begin{equation}\label{eq : condition C1 r0}
      \mathcal{C}_1r_0 <1,
  \end{equation}
 we define $\kappa_1$ as 
$\kappa_1 \bydef\frac{\mathcal{Z}_{u,1}+\mathcal{C}_1r_0}{1-\mathcal{C}_1r_0}.$
If in addition
\begin{equation}\label{eq : condition for kappa1 and C1 r0}
    1-Z_{1,2}-\mathcal{Z}_{u,2} - (1+\kappa_1^2)^\frac{1}{2}\mathcal{C}_1r_0 > 0,
\end{equation}
 then we define
\begin{equation}\label{eq : value for epsilon infinity n}
\begin{aligned}
\kappa_2 &\bydef \frac{Z_{1,1} + (\mathcal{Z}_{u,2} + (1+\kappa_1^2)^\frac{1}{2}\mathcal{C}_1r_0)\|P^N\|_2}{1-Z_{1,2}-\mathcal{Z}_{u,2} - (1+\kappa_1^2)^\frac{1}{2}\mathcal{C}_1r_0}\\
      \epsilon_n^{(\infty)} &\bydef |\lambda_n + t|\left(Z_{1,3} + Z_{1,4} \kappa_2 + \left(\mathcal{Z}_{u,3} + \mathcal{C}_2r_0(1+\kappa_1^2)^\frac{1}{2} \right)\left(\|P^N\|_2 +   \kappa_2 \right)   \right)
\end{aligned}
\end{equation}
for all $n \in \mathbb{Z}^m$
and obtain that $\lambda \in \overline{B_{\epsilon_j^{(\infty)}}(\lambda_j)}$ for some $j \in \mathbb{Z}^m$, where $\lambda_j$ is defined in \eqref{def : lambda n}.
 \end{lemma}
\begin{proof}
    Let $\lambda \in \mathcal{J}\cap \sigma\left(D\mathbb{F}(\tilde{u})\right)$. Then, using Lemma \ref{lem : separation of the spectrum sigma 0}, we have that $\lambda \in \text{Eig}\left(D\mathbb{F}(\tilde{u})\right)$, and we denote $u \in \mathcal{H}$ its associated eigenfunction such that $\|u\|_2=1$. In particular, we have
\begin{align*}
   \mathbb{L}u + D\mathbb{G}(\tilde{u}) u = \lambda u.
\end{align*}
Then, because $\lambda \in \sigma_\delta$, we know that $\mathbb{L} - \lambda I$ is invertible and therefore
\begin{align}\label{eq : eigenvalue compact}
    u +  (\mathbb{L}-\lambda I)^{-1} D\mathbb{G}(\tilde{u}) u = 0. 
\end{align}
This implies that
\begin{align*}
    \|\out u\|_2 &= \left\|\out \left(\mathbb{L} - \lambda I\right)^{-1} D\mathbb{G}(\tilde{u}) u\right\|_2 \\
    &\leq  \left\|\out \left(\mathbb{L} - \lambda I\right)^{-1} \left(D\mathbb{G}(\tilde{u}) - D\mathbb{G}(u_0) \right)u\right\|_2  + \left\|\out \left(\mathbb{L} - \lambda I\right)^{-1} D\mathbb{G}(u_0) u\right\|_2 \\
    & \leq \mathcal{C} r_0( \|\cha u\|_2 +  \|\out u\|_2) + \mathcal{Z}_{u,1} \|\cha u\|_2
\end{align*}
 using that supp$(u_0) \subset \overline{\om}$. In particular, we obtain that
 \begin{equation}\label{eq : norm of u on omega}
      \|\out u\|_2 \leq \frac{\mathcal{Z}_{u,1}+\mathcal{C}_1r_0}{1-\mathcal{C}_1r_0} \|\cha u\|_2 = \kappa_1 \|\cha u \|_2.
 \end{equation}
Now, let us go back to $u = - (\mathbb{L}-\lambda I)^{-1}D\mathbb{G}(\tilde{u}) u$. In particular, defining $\mathbb{K} : L^2 \to L^2$ as $\mathbb{K} \bydef \Gamma^\dagger_d\left((L_d-\lambda)^{-1} DG_d(U_0)\right)$, where $\Gamma^\dagger_d$ is given in \eqref{def : Gamma and Gamma dagger}, we get
\begin{align*}
    u + \mathbb{K}u = \left(\mathbb{K} - (\mathbb{L}-\lambda I)^{-1}D\mathbb{G}(\tilde{u})\right) u,
\end{align*}
which, multiplied by $\cha$, implies
\begin{align*}
    \cha u + \cha \mathbb{K}u = \cha \left(\mathbb{K} - (\mathbb{L}-\lambda I)^{-1}D\mathbb{G}(\tilde{u})\right) u.
\end{align*}
Now, denoting $U \bydef \gamma_d(u) \in \ell^2$ and applying $\gamma_d$, the above equation yields
\begin{align}\label{eq : identity for U as an eigenvector}
   DF(U_0)U-\lambda U= (L-\lambda I) \gamma \left( \left(\mathbb{K} - (\mathbb{L}-\lambda I)^{-1}D\mathbb{G}(\tilde{u})\right) u \right).
\end{align}
Let $P$ be given in \eqref{eq : pseudo diagonalization DF} and recall that $DF(U_0) = P \mathcal{D}P^{-1}$. Denoting $V \bydef P^{-1}U$, the above yields
\begin{align}
  \mathcal{D}V-\lambda V= P^{-1}(L-\lambda I) \gamma \left( \left(\mathbb{K} - (\mathbb{L}-\lambda I)^{-1}D\mathbb{G}(\tilde{u})\right) u \right).
\end{align}
Now, letting $S,R$ be defined in \eqref{eq : pseudo diagonalization DF}, we get
\begin{align}\label{eq : for S and R 1}
  SV-\lambda V= - RV +  P^{-1}(L-\lambda I) \gamma \left( \left(\mathbb{K} - (\mathbb{L}-\lambda I)^{-1}D\mathbb{G}(\tilde{u})\right) u \right).
\end{align}
Let $t \in \mathbb{C}$ be defined as in Lemma \ref{prop : minimum of the spectrum full}, we can transform \eqref{eq : for S and R 1} as 
\small{
\begin{align}\label{eq : for S and R}
  V- (S+ tI)^{-1}(\lambda+t) V= - (S+ tI)^{-1}RV +  (S+ tI)^{-1}P^{-1}(L-\lambda I) \gamma \left( \left(\mathbb{K} - (\mathbb{L}-\lambda I)^{-1}D\mathbb{G}(\tilde{u})\right) u \right).
\end{align}
} \normalsize
Now, since $\pi_N P = \pi_N$ and $\pi_N P^{-1} = \pi_N$, we obtain that 
\[
\pi_N (S+tI)^{-1} = \pi_N (L+tI)^{-1}.
\]
In particular, this implies that 
\[
\pi_N (S+ tI)^{-1}P^{-1}(L-\lambda I) = \pi_N (L+tI)^{-1}(L-\lambda I).
\]
Moreover, we have 
\[
\pi_N\left(V- (S+ tI)^{-1}(\lambda+t) V\right) = \pi_N(I - (L+tI)^{-1}(\lambda+t)) V  = \pi_N (L+tI)^{-1}(L -\lambda I)  V.
\]
Multiplying \eqref{eq : for S and R} by $\pi_N$ and using the above, we get
\begin{align}\label{eq : tail for V}
    \pi_N  V = - \pi_N (L -\lambda I)^{-1}RV +    \gamma \left( \left(\mathbb{K} - (\mathbb{L}-\lambda I)^{-1}D\mathbb{G}(\tilde{u})\right) u \right).
\end{align}
Now, \eqref{eq : tail for V} and the fact that $\|u\|_2 = 1$ imply that 
\small{
\begin{align}\label{eq : step 5 in the proof}
    \|\pi_NV\|_2 &\leq \|\pi_N (L -\lambda I)^{-1}R \pi_N\| \|\pi_NV\|_2 + \|\pi_N (L -\lambda I)^{-1}R\pi^N\|_2 \|\pi^NV\|_2 + \mathcal{Z}_{u,2}\|U\|_2 + \frac{\mathcal{C}_1r_0}{\sqrt{|\om|}}\\
    &\leq Z_{1,1}\|\pi^NV\|_2 + Z_{1,2}\|\pi_NV\|_2 + \mathcal{Z}_{u,2}\|U\|_2 + \frac{(\|\cha u\|_2^2 + \|\out u\|_2^2)^\frac{1}{2}}{\sqrt{|\om|}}\mathcal{C}_1r_0.
\end{align}
}\normalsize
Using \eqref{eq : norm of u on omega} and Parseval's identity, we have 
\begin{align*}
    \frac{(\|\cha u\|_2^2 + \|\out u\|_2^2)^\frac{1}{2}}{\sqrt{|\om|}}\mathcal{C}_1r_0 \leq  (1+\kappa_1^2)^\frac{1}{2} \mathcal{C}_1r_0 \|U\|_2.
\end{align*}
Going back to \eqref{eq : step 5 in the proof}, we get
\begin{equation}\label{eq : ineq for tail of V}
    \|\pi_NV\|_2 \leq \frac{Z_{1,1}}{1-Z_{1,2}} \|\pi^NV\|_2 + \frac{\mathcal{Z}_{u,2}\|U\|_2}{1-Z_{1,2}} + \frac{(1+\kappa_1^2)^\frac{1}{2}\mathcal{C}_1r_0 \|U\|_2}{1-Z_{1,2}}.
\end{equation}
But now, we have that $U = PV$, and therefore
\begin{align}\label{eq : estimate norm of U with V}
  \|U\|_2   \leq 
  \|P^N\|_2\|\pi^NV\|_2 + \|\pi_NV\|_2
\end{align}
where we used that $\pi_N P = \pi_N$. Combined with \eqref{eq : ineq for tail of V}, we get
\begin{align}\label{eq : last estimate tail for V}
    \|\pi_NV\|_2 \leq \frac{Z_{1,1} + (\mathcal{Z}_{u,2} + (1+\kappa_1^2)^\frac{1}{2}\mathcal{C}_1r_0)\|P^N\|_2}{1-Z_{1,2}-\mathcal{Z}_{u,2} - (1+\kappa_1^2)^\frac{1}{2}\mathcal{C}_1r_0} \|\pi^NV\|_2 = \kappa_2 \|\pi^N V\|_2.
\end{align}
This also implies that 
\small{
\begin{align}\label{eq : inequalities between u and v}
    \|U\|_2 \leq (\|P^N\|_2 + \kappa_2) \|\pi^NV\|_2 \text{ and }  \frac{(\|\cha u\|_2^2 + \|\out u\|_2^2)^\frac{1}{2}}{\sqrt{|\om|}} \leq (1+\kappa_1^2)^\frac{1}{2} (\|P^N\|_2 + \kappa_2) \|\pi^NV\|_2.
\end{align}
}\normalsize
In particular, since $\|u\|_2 =1$, we obtain that $\|\pi^NV\|_2 >0$.
Now, multiplying \eqref{eq : for S and R} by $\pi^N$, we get
\footnotesize{
\begin{align}\label{eq : finite part V}
     &(I- (S+ tI)^{-1}(\lambda+t))\pi^N V \\
     = - &\pi^N(S+ tI)^{-1}R\pi^N V - \pi^N(S+ tI)^{-1}R\pi_N V +  \pi^N (S+ tI)^{-1}P^{-1}(L-\lambda I) \gamma \left( \left(\mathbb{K} - (\mathbb{L}-\lambda I)^{-1}D\mathbb{G}(\tilde{u})\right) u \right).
\end{align}
} \normalsize
But now, notice that $(S+ tI)^{-1} : \ell^2 \to \ell^2$ is diagonal and compact. Consequently, we obtain
\begin{align*}
   \min_{\lambda_i \in \sigma(S)} \left|1 - \frac{\lambda+t}{\lambda_i+  t} \right| \|\pi^N V\|_2 \leq \| (I- (S+ tI)^{-1}(\lambda+t))\pi^N V\|_2.
\end{align*}
Going back to \eqref{eq : finite part V} and using \eqref{eq : inequalities between u and v}, we get
\begin{align*}
     &\min_{\lambda_i \in \sigma(S)} \left|1 - \frac{\lambda+t}{\lambda_i+  t} \right| \|\pi^N V\|_2 \\
     \leq & ~ Z_{1,3}\|\pi^N V\|_2 + Z_{1,4}\|\pi_NV\|_2 + \mathcal{Z}_{u,3}  \|U\|_2 + \frac{(\|\cha u\|^2_2 + \|\out u\|_2^2)^\frac{1}{2}}{\sqrt{|\om|}}\mathcal{C}_2r_0\\
     \leq & ~ \left(Z_{1,3} + Z_{1,4} \kappa_2 + \left(\mathcal{Z}_{u,3} + \mathcal{C}_2r_0(1+\kappa_1^2)^\frac{1}{2} \right)\left(\|P^N\|_2 +   \kappa_2 \right)   \right) \|\pi^NV\|_2
\end{align*}
where we also used \eqref{eq : estimate norm of U with V} and \eqref{eq : last estimate tail for V}. This concludes the proof.
\end{proof}

\begin{remark}
    Under Assumptions \ref{ass:A(1)} and \ref{ass : LinvG in L1}, Theorem 3.9 in \cite{unbounded_domain_cadiot} provides that there exists $C>0$ such that $\mathcal{Z}_{u,1}, \mathcal{Z}_{u,2} \leq  C e^{-2\pi \rho d}$, where $\rho$ is given in Assumption \ref{ass:A(1)} and $C$ only depends on $u_0$. In particular, Theorem 3.9 in \cite{unbounded_domain_cadiot} provides an explicit formula for the value of $C$. Moreover, the bounds $Z_{1,i}$ ($i \in \{1,2,3,4\}$) involve the computation of infinite matrix norms, which can be obtained thanks to a meticulous combination of rigorous numerics and analytical estimates. Such computations are recurrent in computer-assisted analysis and we refer the reader to  \cite{unbounded_domain_cadiot, sh_cadiot, period_kuramoto, periodic_navier_spontaneous} for examples. Consequently, computer-assisted analysis allows to derive an explicit value for each $\epsilon_n^{(\infty)}$ in \eqref{eq : value for epsilon infinity n}.  We illustrate such an analysis in Section \ref{sec : application proof stability}.
\end{remark}

The above result will allow to control the spectrum of $sD\mathbb{F}(\tilde{u}) + (1-s)D\mathbb{F}(u_0)$ ($s \in [0,1]$) thanks to the spectrum of $DF(U_0)$. Now, we want to be able to continuously pass from $D\mathbb{F}(u_0)$ to $DF(U_0)$. For this purpose, we define $U^{(q)}_0$ as 
\begin{align*}
    U^{(q)}_0 \bydef \gamma_q(u_0) \in X_q, 
\end{align*}
where $\gamma_q$ is defined in \eqref{eq : gamma and gamma dagger functions},
and we control the eigenvalues of $DF_q(\Uoq)$ thanks to the one of $DF_d(U_0)$. This is achieved in the next lemma.
\begin{lemma}\label{lem : link of spectrum Fourier coeff in q}
  Let $\delta>0$, $q \geq d$, $\mathcal{J}$ be a Jordan domain in $\overline{\sigma_\delta}$  and let $\lambda \in \mathcal{J}\cap \sigma\left(DF_q\left(\Uoq\right)\right)$. Now, let $\mathcal{Z}_{u,1}^{(q)}, \mathcal{Z}^{(q)}_{u,2}$  and $\mathcal{Z}^{(q)}_{u,3}$ be constants satisfying the following
  \begin{align*}
     \mathcal{Z}_{u,1}^{(q)} &\geq  \sup_{\mu \in \mathcal{J}} \left\|\mathbb{1}_{\Omega_q \setminus \om} \left((\mathbb{L}- \mu I)^{-1} - \Gamma^\dagger_q\left(  ({L}_q- \mu I)^{-1}\right)\right) D{\mathbb{G}}(u_0) \right\|_2 + \left\|\mathbb{1}_{\Omega_q \setminus \om} (\mathbb{L}- \mu I)^{-1} D{\mathbb{G}}(u_0)\right\|_2\\
     \mathcal{Z}_{u,2}^{(q)} &\geq \sup_{\mu \in \mathcal{J}} \| \cha \left(\Gamma^\dagger_q\left((L_q- \mu I )^{-1}\right) - \Gamma^\dagger_d\left((L_d- \mu I )^{-1}\right)\right)D\mathbb{G}(u_0) \|_2\\
     \mathcal{Z}_{u,3}^{(q)} &\geq \sup_{\mu \in \mathcal{J}} \|\pi^N (S+ tI)^{-1}P^{-1}(L-\mu I)\|_2 \| \cha \left(\Gamma^\dagger_q\left((L_q- \mu I )^{-1}\right) - \Gamma^\dagger_d\left((L_d- \mu I )^{-1}\right)\right)D\mathbb{G}(u_0) \|_2.
  \end{align*}
  Moreover, let $Z_{1,i}$ ($i \in \{1,2,3,4\}$) be the constants introduced in Lemma \ref{lem : link of spectrum unbounded to U0}.
  If
  \begin{equation}\label{eq : condition for Zu2 q}
      1 - Z_{1,2} - \mathcal{Z}^{(q)}_{u,2} > 0,
  \end{equation}
 then we define
\begin{equation}\label{eq : value for epsilon q n}
      \epsilon_n^{(q)} \bydef |\lambda_n + t|\left(Z_{1,3} + Z_{1,4} \frac{Z_{1,1} + \mathcal{Z}^{(q)}_{u,2}\|P^N\|_2}{1-Z_{1,2}-\mathcal{Z}_{u,2}^{(q)}} +  \mathcal{Z}_{u,3}^{(q)}\left(\|P^N\|_2 +   \frac{Z_{1,1} + \mathcal{Z}_{u,2^{(q)}}\|P^N\|_2}{1-Z_{1,2}-\mathcal{Z}_{u,2}^{(q)}} \right) \right) 
\end{equation}
for all $n \in \mathbb{Z}^m$
and obtain that $\lambda \in \overline{B_{\epsilon_j^{(q)}}(\lambda_j)}$ for some $j \in \mathbb{Z}^m$, where $\lambda_j$ is given in \eqref{def : lambda n}.
 \end{lemma}

\begin{proof}
    Let $\lambda \in \mathcal{J}\cap \sigma\left(D{F}_q(\Uoq)\right)$. Then, using the proof of Lemma  \ref{eq : spectrum of DFU0 is point}, we have that $\lambda \in \text{Eig}\left(D{F}_q(\Uoq)\right)$, and we denote $U_q \in X_q$ its associated eigenvector such that $\sqrt{|\omq|}\|U_q\|_2=1$. In particular, we have
\begin{align*}
   L_q U_q + D{G}_q(\Uoq) U_q = \lambda U_q.
\end{align*}
Then, because $\lambda \in \sigma_\delta$, we know that ${L}_q - \lambda I$ is invertible and therefore
\begin{align}\label{eq : eigenvalue compact for q}
    U_q +  ({L}_q-\lambda I)^{-1} D{G}_q(\Uoq) U_q = 0. 
\end{align}
Using \eqref{eq : gamma and gamma dagger functions}, this implies that
\small{
\begin{align*}
  &\|\mathbb{1}_{\Omega_q \setminus \om}\gamma^\dagger_q(U_q)\|_2 =\left\|\mathbb{1}_{\Omega_q \setminus \om} \Gamma^\dagger_q\left(  ({L}_q-\lambda I)^{-1}\right) D{\mathbb{G}}(u_0) \gamma_q^\dagger(U_q)\right\|_2 \\
   \leq & ~\left\|\mathbb{1}_{\Omega_q \setminus \om} \left((\mathbb{L}-\lambda I)^{-1} - \Gamma^\dagger_q\left(  ({L}_q-\lambda I)^{-1}\right)\right) D{\mathbb{G}}(u_0) \gamma_q^\dagger(U_q)\right\|_2 + \left\|\mathbb{1}_{\Omega_q \setminus \om} (\mathbb{L}-\lambda I)^{-1} D{\mathbb{G}}(u_0) \gamma_q^\dagger(U_q)\right\|_2\\
    \leq & ~ \mathcal{Z}_{u,1}^{(q)}  \|\cha \gamma^\dagger_q(U_q)\|_2
\end{align*}
} \normalsize
 using that supp$(u_0) \subset \overline{\om}$. In particular, using Parseval's identity we obtain that $$\|\cha \gamma^\dagger_q(U_q)\|_2^2 = 1 -  \|\mathbb{1}_{\Omega_q \setminus \om}\gamma^\dagger_q(U_q)\|_2^2  \geq 1 - (\mathcal{Z}_{u,1}^{(q)} )^2\|\cha \gamma^\dagger_q(U_q)\|_2^2.$$ 
 This implies that 
\begin{align}\label{eq : lower bound norm u in proof for q}
    \|\cha \gamma^\dagger_q(U_q)\|_2 \geq \frac{1}{\sqrt{1+(\mathcal{Z}_{u,1}^{(q)} )^2}} >0.
\end{align}
Now, let us define $U = \gamma_d(\gamma_q^\dagger(U_q))$. Going back to \eqref{eq : eigenvalue compact for q} and using a similar reasoning as for the proof of Lemma \ref{lem : link of spectrum unbounded to U0}, we apply $\gamma_d(\gamma^\dagger_q(\cdot))$ to \eqref{eq : eigenvalue compact for q} and get
\begin{align*}
    DF_d(U_0)U - \lambda U = (L_d-\lambda)\gamma_d \left( \left(\Gamma^\dagger_d\left((L_d-\lambda)^{-1} DG_d(U_0)\right) -  \Gamma^\dagger_q\left(({L}_q-\lambda I)^{-1} D{G}_q(\Uoq)\right)\right) U \right).
\end{align*}
We finish the proof following the steps of the proof of Lemma \ref{lem : link of spectrum unbounded to U0}. 
\end{proof}
Finally, combining Lemma \ref{lem : link of spectrum unbounded to U0} and \ref{lem : link of spectrum Fourier coeff in q}, we derive our main result based on the construction of an homotopy from $\mathcal{D}$ to  $D\mathbb{F}(\tilde{u})$.
\begin{theorem}\label{th : gershgorin unbounded}
    Let $\delta>0$, $\mathcal{J}$ be a closed Jordan domain in $\overline{\sigma_\delta}$ and let $t \in \mathbb{C}$ be given as in Lemma \ref{prop : minimum of the spectrum full}. Moreover, suppose that \eqref{eq : condition C1 r0}, \eqref{eq : condition for kappa1 and C1 r0} and \eqref{eq : condition for Zu2 q} are satisfied. Then for all $n \in \mathbb{Z}^m$, define $\epsilon_n$ as 
    \begin{align*}
        \epsilon_n \bydef \max\left\{\sup_{q \in [d,\infty)} \epsilon^{(q)}_n, \epsilon^{(\infty)}_n\right\}
    \end{align*}
    where $\epsilon^{(q)}_n$ is given in \eqref{eq : value for epsilon q n} for all $q\in [d,\infty)$ and $\epsilon^{(\infty)}_n$ is given in \eqref{eq : value for epsilon infinity n}.
    Let  $k \in \mathbb{N}$ and $I \subset \mathbb{Z}^m$ such that $|I| = k$. If $\cup_{n \in I} \overline{B_{\epsilon_n}(\lambda_n)} \subset \mathcal{J}$ and $\left(\cup_{n \in I} \overline{B_{\epsilon_n}(\lambda_n)}\right) \bigcap\left( \cup_{n \in \mathbb{Z}^m \setminus I} \overline{B_{\epsilon_n}(\lambda_n)}\right) = \varnothing$, then there are exactly $k$ eigenvalues of $D\mathbb{F}(\tilde{u})$ in $\cup_{n \in I} \overline{B_{\epsilon_n}(\lambda_n)} \subset  \mathcal{J}$ counted with multiplicity.
\end{theorem}
\begin{proof}
First, let $\mathcal{D}(s)$ be defined as $\mathcal{D}(s) \bydef (1-s)PSP^{-1} + s DF(U_0)$ for all $s \in [0,1]$. Then, using Lemma \ref{lem : gershgorin matrix}, we obtain that there are exactly $k$ eigenvalues of $\mathcal{D}(s)$ in $\cup_{n \in {I}} B_{\epsilon}(\lambda_n) \subset \mathcal{J}$ counted with multiplicity for all $s \in [0,1]$.

In the rest of the proof, we abuse notation and denote $DF_\infty = D\mathbb{F}$, as well as $U_0^{(\infty)} = u_0$.
Now, we prove that $q \to \Gamma^\dagger_q\left((DF_q(\Uoq) + tI)^{-1}\right) : L^2 \to L^2$ is continuous on $[d,\infty]$.  Let $q_1, q_2 \in [d,\infty]$, then 
\begin{align*}
    &\|\Gamma^\dagger\left((DF_{q_1}(U_0^{(q_1)}) + tI)^{-1}\right) - \Gamma^\dagger\left((DF_{q_2}(U_0^{(q_2)}) + tI)^{-1}\right)\|_2 \\
    \leq  &~\|\Gamma^\dagger_{q_1}\left((I + L_{q_1}^{-1}DG_{q_1}(U_0^{(q_1)}) + tI)^{-1}\right)\left( \Gamma^\dagger_{q_1}(L_{q_1}^{-1}) -  \Gamma^\dagger_{q_2}(L_{q_2}^{-1})\right)\|_2\\
    & ~~+ \|\left(\Gamma^\dagger_{q_1}\left((I + L_{q_1}^{-1}DG_{q_1}(U_0^{(q_1)}) + tI)^{-1}\right) - \Gamma^\dagger_{q_2}\left((I + L_{q_2}^{-1}DG_{q_2}(U_0^{(q_2)}) + tI)^{-1}\right) \right)   \Gamma^\dagger_{q_2}(L_{q_2}^{-1})\|_2.
\end{align*}
Consequently, we will prove our result if we prove that 
\[
q \to \Gamma_q^\dagger(L^{-1}_q) : L^2 \to L^2  \text{ and } q \to \Gamma_q^\dagger(L^{-1}_q DG_q(\Uoq)) : L^2 \to L^2
\]
are continuous. First, we clearly have that $DG_q(\Uoq) : X_q \to \ell^2$ is continuous since supp$(u_0) \subset \overline{\om}$ and $u_0 \in \mathcal{H}$. Consequently, it remains to prove that $q \to \Gamma_q^\dagger(L^{-1}_q) : L^2 \to L^2 $ is continuous. First, we prove that $q \to \Gamma_q^\dagger(L^{-1}_q) : L^2 \to L^2 $ is continuous on $[d,b]$ for any $b \geq d$. Indeed, let $b \in [d,\infty)$ and let $q \in [d,b]$. Then, observe that 
\begin{align*}
    (L^{-1}_q U)_n = \frac{1}{l(\frac{\pi}{q}n)} U_n
\end{align*}
by definition of $L^{-1}_q$. Since $\frac{1}{l} : \R^m \to \mathbb{C}$ is smooth, we have that $q \to \frac{1}{l(\frac{\pi}{q}n)}$ is uniformly continuous on $[d,b]$, hence we obtain the continuity of $q \to \Gamma_q^\dagger(L^{-1}_q) : L^2 \to L^2 $ on $[d,b]$. Finally, it remains to prove that 
\begin{align*}
   \lim_{q \to \infty} \|\Gamma_q^\dagger(L^{-1}_q) - \mathbb{L}^{-1}\|_2 =0. 
\end{align*}
Let $\phi \in C^\infty_0(\R^m)$, that is $\phi$ is infinitely differentiable and has a compact support on some bounded domain, say $\mathcal{K}$. Using \cite{unbounded_domain_cadiot}, one can prove that there exists $a >0$ and $C$ depending on $\mathcal{K}$ only such that 
\begin{align*}
    \|\Gamma_q^\dagger(L^{-1}_q)\phi - \mathbb{L}^{-1}\phi\|_2 \leq C e^{-aq}
\end{align*}
for all $q$ big enough so that $\mathcal{K} \subset \omq$. Consequently, given now  $\phi \in L^2$, we can use a sequence $(\phi_n)_{n \in \mathbb{N}}$ such that $\phi_n \in  C^\infty_0(\R^m)$ for all $n \in \mathbb{N}$ and such that $\|\phi - \phi_n\|_2 \to 0$ as $n \to \infty$. Using the above, this implies that 
\begin{align*}
   \lim_{q \to \infty} \|\Gamma_q^\dagger(L^{-1}_q) - \mathbb{L}^{-1}\|_2 =0. 
\end{align*}
and that $q \to \Gamma^\dagger_q\left((DF_q(\Uoq) + tI)^{-1}\right) : L^2 \to L^2$ is continuous on $[d,\infty]$. In particular, $\left(DF_q(\Uoq) + tI\right)^{-1} : \ell^2 \to \ell^2$ is a Fredholm operator for all $q \in [d,\infty]$ (see the proofs of  Lemmas \ref{lem : separation of the spectrum sigma 0} and \ref{eq : spectrum of DFU0 is point}), and therefore $\sigma(DF_q(\Uoq) + tI)^{-1}) \cap \mathcal{J}$  consists of a finite number of eigenvalues counted with multiplicity.  Now, given $q \in [d,\infty]$ and  $\lambda \in \sigma(DF_q(\Uoq)) \cap \mathcal{J}$, Lemmas \ref{lem : link of spectrum unbounded to U0} and \ref{lem : link of spectrum Fourier coeff in q} provide that $\lambda \in \overline{B_{\epsilon_j}(\mu_j)}$ for some $\mu_j \in \sigma(S)$.
Using Chapter 4.3 of  \cite{kato2013perturbation}, we obtain that there are exactly $k$ eigenvalues of $DF_q(\Uoq)$ in $\cup_{n \in {I}} \overline{B_{\epsilon_n}(\lambda_n)} \subset \mathcal{J}$ counted with multiplicity for all $q \in [d,\infty]$.

    Finally, let $\mathbb{M}(s)$ be defined as $\mathbb{M}(s) \bydef sD\mathbb{F}(\tilde{u}) + (1-s)D\mathbb{F}(u_0)$ for all $s \in [0,1]$. Then, since $D\mathbb{F}(v) : \mathcal{H} \to L^2$ is smooth for all $v \in \mathcal{H}$, we have that $\mathbb{M}$ is smooth as a function of $s$. In particular, we have that $s \to (\mathbb{M}(s) + tI)^{-1} : L^2 \to L^2$ is continuous.
     Using Lemma \ref{lem : link of spectrum unbounded to U0}, given $\lambda \in \sigma(\mathbb{M}(s)) \cap \mathcal{J}$, we have that $\lambda \in \overline{B_{\epsilon_j}(\mu_j)}$ for some $\mu_j \in \sigma(S)$. Using Chapter 4.3 of  \cite{kato2013perturbation} again, combined with the above, we obtain that there are exactly $k$ eigenvalues of $\mathbb{M}(s)$ in $\cup_{n \in {I}} \overline{B_{\epsilon_n}(\lambda_n)} \subset \mathcal{J}$ counted with multiplicity for all $s \in [0,1]$. In particular, $\mathbb{M}(1) = D\mathbb{F}(\tilde{u})$, which concludes the proof.
\end{proof}

\begin{remark}
    In the case where $DF(U_0) : \ell^2 \to \ell^2$, $D\mathbb{F}(u_0) : L^2 \to L^2$ and   $D\mathbb{F}(\tilde{u}) : L^2 \to L^2$ are self-adjoint, then we can simplify the bounds of Lemma \ref{lem : link of spectrum unbounded to U0} and use Lemma \ref{lem : gershgorin matrix}. Indeed, given $\lambda \in \sigma(D\mathbb{F}(\tilde{u}))\cap \mathcal{J}$ and using \eqref{eq : identity for U as an eigenvector}, we have 
    \begin{align*}
         U-(\lambda+t) (DF(U_0)+tI)^{-1} U= (DF(U_0)+tI)^{-1}(L-\lambda I) \gamma \left( \left(\mathbb{K} - (\mathbb{L}-\lambda I)^{-1}D\mathbb{G}(\tilde{u})\right) u \right).
    \end{align*}
    Since $(DF(U_0)+tI)^{-1} :\ell^2 \to \ell^2$ is self-adjoint and compact, we use the spectral theorem to get
    \begin{align*}
        \min_{\mu_n \in \sigma(DF(U_0))} \left|\frac{\lambda-\mu_n}{\mu_n + t}\right| \|U\|_2 \leq \widetilde{\mathcal{Z}}_{u,3}\|U\|_2 + \widetilde{\mathcal{C}}_2r_0(1+\kappa_1^2)^{\frac{1}{2}}\|U\|_2,
    \end{align*}
    where $\widetilde{\mathcal{Z}}_{u,3}$ and $\widetilde{C}_2$ satisfy
    \begin{align*}
       \widetilde{\mathcal{Z}}_{u,3} & \geq   \sup_{\mu \in \mathcal{J}} \|(DF(U_0)+tI)^{-1}(L- \mu I)\|_2 \| \cha \left(\Gamma^\dagger\left((L- \mu I )^{-1}\right) - (\mathbb{L}- \mu I)^{-1}\right)D\mathbb{G}(u_0) \|_2\\
       \widetilde{C}_{2} &\geq  \frac{1}{r_0} \sup_{\mu \in \mathcal{J}} \|(DF(U_0)+tI)^{-1}(L- \mu I)\|_2 \left\| \left(\mathbb{L} - \mu I\right)^{-1} \left(D\mathbb{G}(u_0)-D\mathbb{G}(\tilde{u})\right)\right\|_2.
    \end{align*}
    Now, using that $\mu_n \in \overline{B_{r_n}(\lambda_n)}$, where $r_n$ is given in Lemma \ref{lem : gershgorin matrix}, we obtain 
    \begin{align}
        |\lambda- \lambda_n| \leq r_n + \left(\widetilde{\mathcal{Z}}_{u,3} + \widetilde{\mathcal{C}}_2r_0(1+\kappa_1^2)^{\frac{1}{2}}\right)(r_n + |\lambda_n+t|).
    \end{align}
    Similarly, we can simplify the bounds of Lemma \ref{lem : link of spectrum Fourier coeff in q}. In fact, given $\lambda \in \sigma(DF_q(U_0^{(q)}))\cap \mathcal{J}$, we obtain
    \begin{align*}
         |\lambda- \lambda_n| \leq r_n + \widetilde{\mathcal{Z}}^{(q)}_{u,3} (r_n + |\lambda_n+t|)
    \end{align*}
    where $\widetilde{\mathcal{Z}}^{(q)}_{u,3}$ satisfies 
    \small{
    \begin{align}
       \widetilde{\mathcal{Z}}^{(q)}_{u,3} &\geq \sup_{\mu \in \mathcal{J}}  \|(DF(U_0)+tI)^{-1}(L- \mu I)\|_2 \| \| \cha \left(\Gamma^\dagger_q\left((L_q- \mu I )^{-1}\right) - \Gamma^\dagger_d\left((L_d- \mu I )^{-1}\right)\right)D\mathbb{G}(u_0) \|_2.
    \end{align}
    } \normalsize
\end{remark}

\begin{remark}
    Note that the main goal of Theorem \ref{th : gershgorin unbounded} is to locate the eigenvalues of $D\mathbb{F}(\tilde{u})$. In order to obtain sharper enclosures of the eigenvalues (and associated eigenfunctions), one can use a fixed-point argument, as presented in Section 5.3 of \cite{unbounded_domain_cadiot}, in a second time. In the self-adjoint case, one can also use the results of Section 10 from \cite{plum2019numerical}. 
\end{remark}

\begin{remark}
    In Theorem 3.17 of \cite{sh_cadiot}, we established that localized solutions can be proven, under some criteria to satisfy, to be the limit of a continuous family of spatially periodic solutions to \eqref{eq : equation intro} as the period tends to infinity.   In particular, the solutions investigated in Section \ref{sec : application proof stability} were all established together with their respective branch of periodic solutions. Using the proof of Theorem \ref{th : gershgorin unbounded} and the homotopy in the parameter $q$, we obtain that Theorem \ref{th : gershgorin unbounded} also provides a localization for the spectrum of the Jacobian at  each periodic solution of the branch. In particular, the obtained stability results in Section \ref{sec : application proof stability} for the localized solutions are applicable to the related periodic solutions as well (under a single application of Theorem \ref{th : gershgorin unbounded}). 
\end{remark}

\section{Application to the study of stability}\label{sec : application proof stability}

In this section we present applications of the framework described in the previous sections. In particular, we study the stability of localized solutions in different PDE and nonlocal models. We provide details for the computation of the bounds of Lemma \ref{lem : link of spectrum unbounded to U0} and \ref{lem : link of spectrum Fourier coeff in q} for specific applications. Then, explicit values for the bounds are obtained thanks to the use of \emph{interval arithmetic} \cite{Moore_interval_analysis}. Using the Julia packages \emph{IntervalArithmetic.jl} \cite{julia_interval} and \emph{RadiiPolynomial.jl}  \cite{julia_olivier}, we implement the rigorous computation of the bounds in \cite{julia_cadiot}. Moreover, we use the Matlab package \emph{Intlab} \cite{Rump} for the computation of matrix products and matrix norms. Indeed, Intlab is optimized for rigorous vector and matrix computations. On the other hand, the sequence structures implemented in \cite{julia_olivier} provide a natural environment for the usage of Fourier coefficients under different symmetries and tensor products.

\subsection{The planar Swift-Hohenberg PDE}\label{ssec : application SH}

We first investigate stability of stationary localized solutions, also denoted localized patterns in this context, in the planar Swift-Hohenberg PDE :
\begin{equation}\label{eq : swift original}
      u_t = -\left((I_d+\Delta)^2u +  \mu u +  \nu_1 u^2 +  \nu_2u^3\right) \bydef  \mathbb{F}(u), \quad u = u(x,t), ~~ x \in \R^2,
\end{equation}
where $\mu >0$ and $(\nu_1, \nu_2) \in \R^2$ are given parameters. In particular, given a stationary localized solution $\tilde{u} : \R^2 \to \R$ of \eqref{eq : swift original}, that is $\mathbb{F}(\tilde{u})=0$ and $\tilde{u}(x) \to 0$ as $|x|_2 \to \infty$, we are interested in the stability of $\tilde{u}$. Naturally, such a property can be studied thanks to the spectrum of $D\mathbb{F}(\tilde{u}) : \mathcal{H} \to L^2(\R^2)$. 

Using \cite{sh_cadiot} and the related data, we have access to constructive existence proofs of localized patterns associated to approximate solutions satisfying \eqref{eq : approximate solution} and \eqref{eq : defect with true solution}. In fact, the solutions established in \cite{sh_cadiot} are invariant under reflection about the $x$ and $y$ axis. That is $\tilde{u},u_0 \in \mathcal{H}_{cc}$, where 
\begin{align*}
    \mathcal{H}_{cc} \bydef \left\{u \in \mathcal{H}, ~ u(x_1,x_2) = u(-x_1,x_2) = u(x_1,-x_2) \text{ for all } x \in \R^2\right\}.
\end{align*}
In particular, $u_0$ has a cosine-cosine series representation. Moreover, note that $D\mathbb{F}(\tilde{u}) : L^2 \to L^2$ is self-adjoint, and consequently its spectrum is real-valued. Moreover, we have 
\begin{equation}
    \mathbb{L} \bydef -(I+\Delta)^2 -  \mu I \text{ and } l(\xi) = - (1-|2\pi\xi|_2^2)^2 - \mu 
\end{equation}
for all $\xi \in \R^2$. In particular, we readily have that 
\begin{equation}\label{eq : essential sh}
    \sigma_{ess}(D\mathbb{F}(\tilde{u})) = (-\infty,-\mu].
\end{equation}
Given $\delta>0$, and combining \eqref{def : sigma delta} and \eqref{eq : essential sh}, we can define $\sigma_\delta$ as
\begin{equation}
    \sigma_\delta = (\delta - \mu , \infty).
\end{equation}
Note that we only consider the restriction of $\sigma_\delta$ on $\R$ since the spectrum of $D\mathbb{F}(\tilde{u})$ is real. 
%
%
%
%
In practice, we notice $\mathcal{H}$ can be decomposed as follows
\begin{align*}
\mathcal{H} = \mathcal{H}_{cc} \oplus \mathcal{H}_{sc} \oplus \mathcal{H}_{cs} \oplus \mathcal{H}_{ss},
\end{align*}
where 
\begin{align*}
    \mathcal{H}_{sc} &\bydef \left\{u \in \mathcal{H}, ~ u(x_1,x_2) = -u(-x_1,x_2) = -u(x_1,-x_2) \text{ for all } x \in \R^2\right\},\\
    \mathcal{H}_{cs} &\bydef \left\{u \in \mathcal{H}, ~ u(x_1,x_2) = u(-x_1,x_2) = -u(x_1,-x_2) \text{ for all } x \in \R^2\right\},\\
    \mathcal{H}_{ss} &\bydef \left\{u \in \mathcal{H}, ~ u(x_1,x_2) = -u(-x_1,x_2) = -u(x_1,-x_2) \text{ for all } x \in \R^2\right\}.
\end{align*}
In a similar spirit, we decompose $L^2 = L^2_{cc} \oplus L^2_{sc} \oplus L^2_{cs} \oplus L^2_{ss}$, where $L^2_{ij}$ possesses the same symmetries as $\mathcal{H}_{ij}$. Moreover, given $\tilde{u} \in \mathcal{H}_{cc}$, we have that 
\begin{align*}
    D\mathbb{F}(\tilde{u} )v_{ij} \in L^2_{ij}
\end{align*}
for all $v_{ij} \in \mathcal{H}_{ij}$ and all $i,j \in \{c,s\}$. This implies that if $u \in \mathcal{H}$ satisfies 
\begin{align*}
    D\mathbb{F}(\tilde{u} )u = \lambda u
\end{align*}
and we decompose $u = u_{cc} + u_{sc} + u_{cs} + u_{ss}$, where $u_{ij} \in \mathcal{H}_{i,j}$, then 
\[
 D\mathbb{F}(\tilde{u} )u_{ij} = \lambda u_{ij}
\]
for all $i,j \in \{c,s\}$. Consequently, we can investigate the spectrum of $D\mathbb{F}_{ij}(\tilde{u} ) : \mathcal{H}_{ij} \to L^2_{ij}$, the restriction of $D\mathbb{F}(\tilde{u})$ to $\mathcal{H}_{ij} \to L^2_{ij}$, for all $i,j \in \{c,s\}$ and obtain the following disjoint decomposition 
\begin{equation}\label{eq : identity symmetry spectrum SH}
    \sigma\left(D\mathbb{F}(\tilde{u})\right) =  \sigma\left(D\mathbb{F}_{cc}(\tilde{u})\right) \cup  \sigma\left(D\mathbb{F}_{cs}(\tilde{u})\right) \cup  \sigma\left(D\mathbb{F}_{sc}(\tilde{u})\right) \cup  \sigma\left(D\mathbb{F}_{ss}(\tilde{u})\right).
\end{equation}
Noticing that the analysis of Section \ref{sec : control spectrum DF unbounded} applies when restricted to symmetric functions, we control the spectrum of each $D\mathbb{F}_{ij}(\tilde{u})$ and use the identity \eqref{eq : identity symmetry spectrum SH} to conclude.
In practice, this allows to decrease the numerical complexity of the rigorous numerics, since the matrix sizes are approximately decreased by a factor $4$. This is notably useful in the planar Swift-Hohenberg PDE, since localized patterns are highly non-trivial and require a large amount of Fourier coefficients to obtain an accurate approximation. 


Since localized patterns are solutions defined on an unbounded domain, they are subject to natural invariances, which lead to a non-empty kernel.
We describe further the kernel of $D\mathbb{F}(\tilde{u})$ and its properties. Let $SE(2)$ be the usual group of Euclidean symmetries on the plane. In particular, $SE(2)$ is generated by 2D translations and a 1D rotation. Moreover, the set of localized solutions to \eqref{eq : swift original} is invariant under $SE(2)$, which can potentially generate a 3D kernel. In fact, we have the following result.
\begin{lemma}\label{lem : radially symmetric}
  Suppose that $\mathbb{F}(\tu) =0$, that $\tu \in H^\infty(\R^2) \cap \mathcal{H}_{cc} \setminus \{0\}$ and that $\mathbb{F}$ is invariant under the Euclidean group of symmetries $SE(2)$.  Then, we have that  $\partial_{x1}\tilde{u},\partial_{x_2} \tu \neq 0$,  $\partial_{x_2} \tilde{u} \in Ker(D\mathbb{F}_{cs}(\tu))$ and $\partial_{x_1} \tilde{u} \in Ker(D\mathbb{F}_{sc}(\tu))$. Moreover, if  $Ker(D\mathbb{F}_{ss}(\tu)) = \{0\}$, then $\tu$ is radially symmetric.
\end{lemma}
\begin{proof}
Since $\mathbb{F}$ is invariant under translations and rotations on the plane, we readily have that $\partial_{x_1}\tilde{u},\partial_{x_2} \tu , \partial_\theta \tu \in D\mathbb{F}(\tu)$ where $\partial_\theta$ is the angular derivative in polar coordinates. If $\partial_{x_1}\tilde{u} =0$, we would have  $\tilde{u}(\cdot,x_2)$  constant for all $x_2$. Since $\tilde{u} \in H^\infty(\R^2)$, this would imply that $\tilde{u}=0$, which is not true by assumption. The same reasoning applies to $\partial_{x_2}\tilde{u}$ and we get that    $\partial_{x1}\tilde{u},\partial_{x_2} \tu \neq 0$,  $\partial_{x_2} \tilde{u} \in Ker(D\mathbb{F}_{cs}(\tu))$ and $\partial_{x_1} \tilde{u} \in Ker(D\mathbb{F}_{sc}(\tu))$ since $\tilde{u} \in \mathcal{H}_{cc}$. 

Now, $\tilde{u}$  has a polar expansion given as 
       $\tilde{u}(r,\theta) = \sum_{n=0}^\infty \tilde{u}_n(r)\cos(2n\theta).$
    In particular, we obtain that $
       \partial_\theta \tilde{u}(r,\theta) = - \sum_{n=0}^\infty 2n \tilde{u}_n(r)\sin(2n\theta).$
 Consequently, this implies that $\partial_\theta \tilde{u} \in Ker(D\mathbb{F}_{ss}(\tilde{u}))$. In particular, if $Ker(D\mathbb{F}_{ss}(\tu)) = \{0\}$ then $\partial_\theta \tilde{u} =0$ and we obtain that $\tilde{u}$ is radially symmetric. 
\end{proof}
The above result will be useful for determining the dimension of the kernel of $D\mathbb{F}(\tilde{u})$, as well as radial symmetry (cf. Theorem \ref{th : unstable gray scott}).
Note that the smoothness condition $\tilde{u} \in H^\infty(\R^2)$ is ensured by Proposition 2.5 in \cite{unbounded_domain_cadiot}.

\subsubsection{The unstable square pattern}\label{ssec : unstable square pattern}

In this section, we investigate the stability of the localized square pattern $\tilde{u}$, whose existence is established in Theorem 4.1 of \cite{sh_cadiot}. We recall that $\mathbb{F}(\tilde{u}) = 0$ with $\mu = 0.28$, $\nu_1 = -1.6$ and $\nu_2 = 1$. Moreover, we denote $u_0 \in \mathcal{H}_{cc}$ the associated approximate solution (represented in Figure \ref{fig : square pattern}) satisfying \eqref{eq : approximate solution} and \eqref{eq : defect with true solution} with $r_0 = 1.16 \times 10^{-5}$ in that case (cf. Theorem 4.1 in \cite{sh_cadiot}).
\begin{figure}[H]
\centering
 \begin{minipage}{.55\linewidth}
  \centering\epsfig{figure=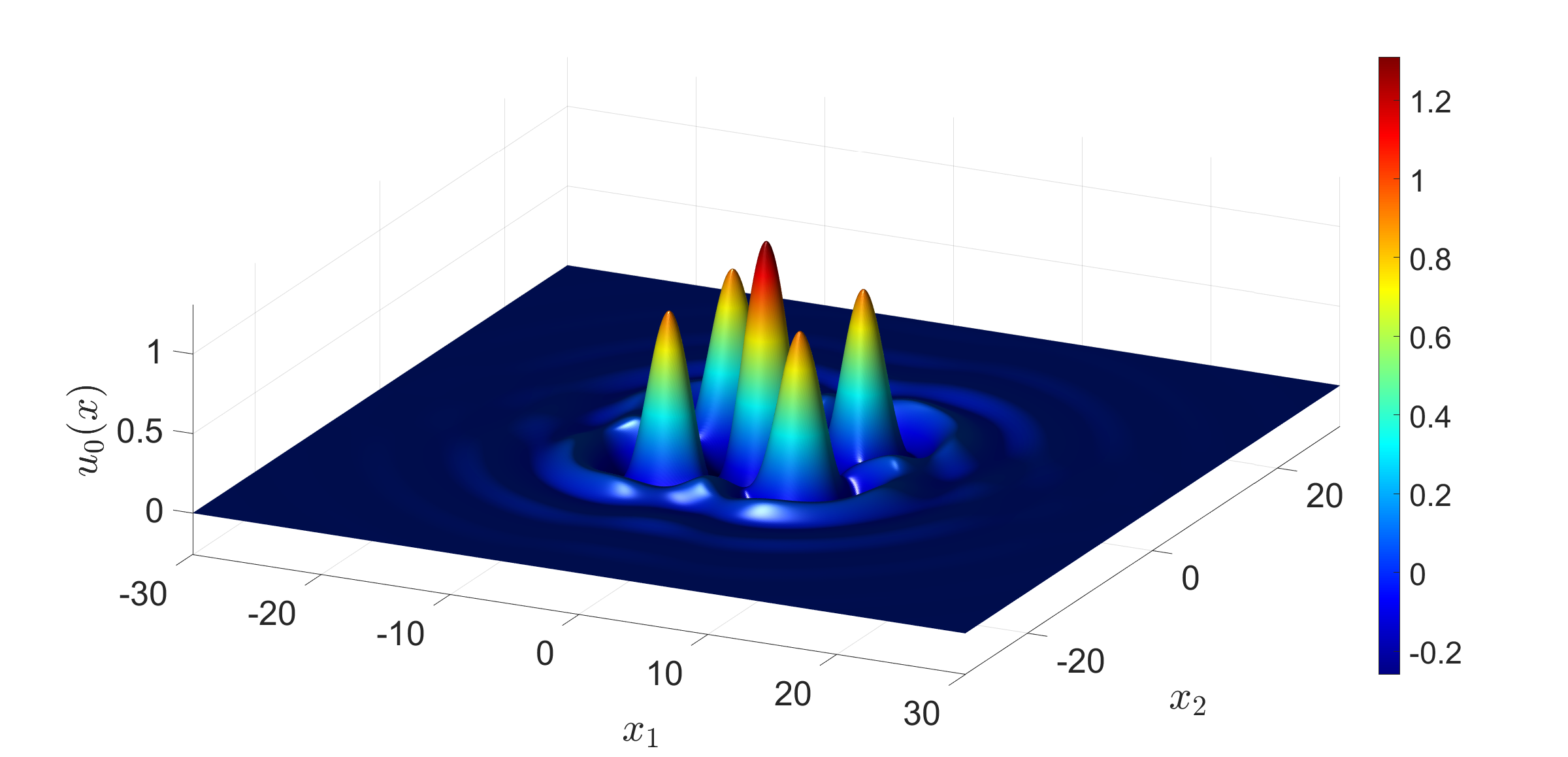,width=\linewidth}
 \end{minipage}%
 \begin{minipage}{.55\linewidth}
  \centering\epsfig{figure=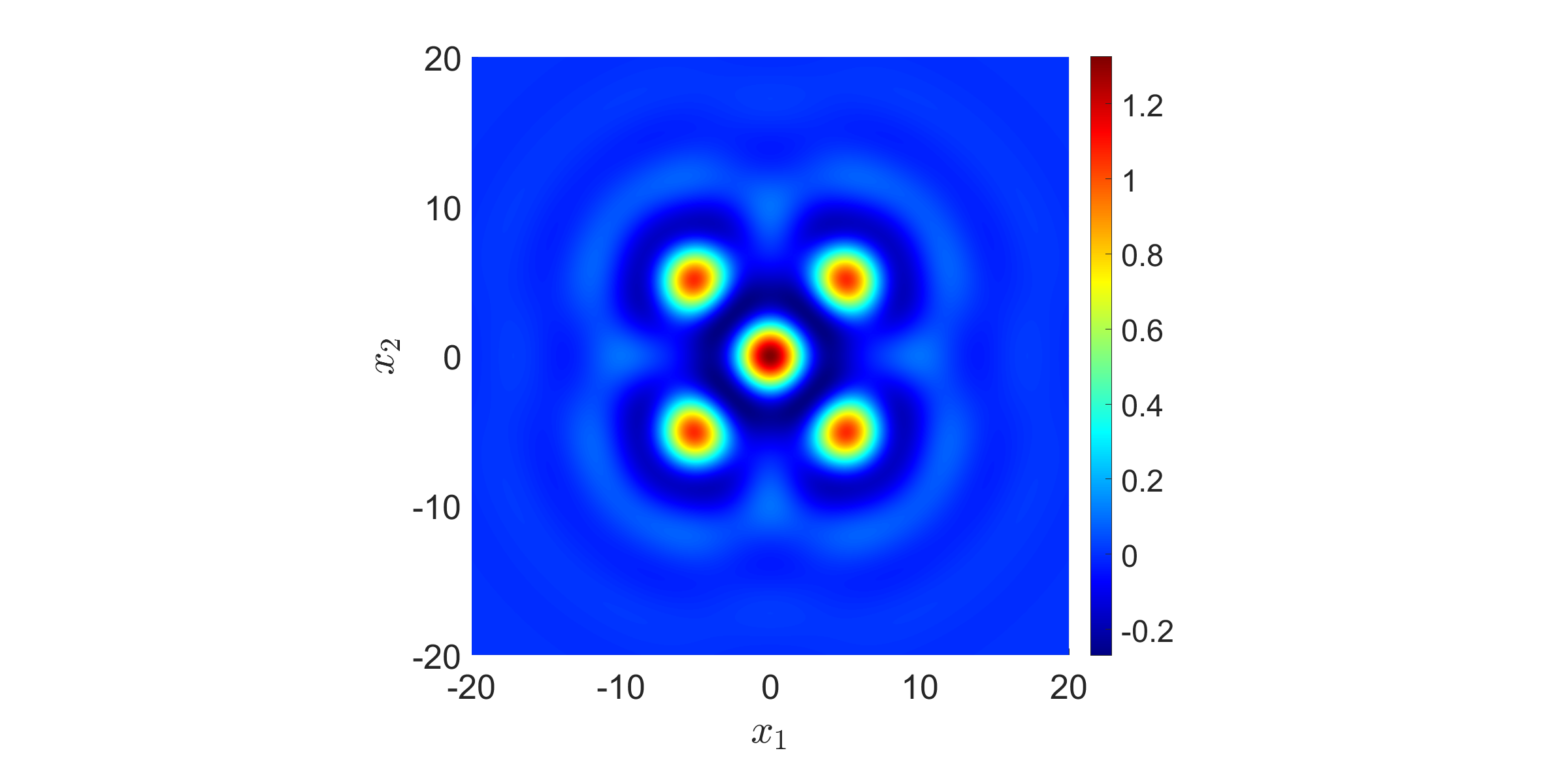,width=\linewidth}
 \end{minipage} 
 \caption{Side and upper views of the approximate solution $u_0$ of the square pattern established in Theorem 4.1 in \cite{sh_cadiot}}\label{fig : square pattern}
 \end{figure}
In what follows, we demonstrate that the framework established in the previous section not only allows to conclude about unstability, but also provides the number of unstable directions with their respective positive eigenvalues. This provides a deeper understanding of the local dynamics surrounding the localized pattern.
\begin{theorem}\label{th : unstable square solution}
    Let $\tilde{u}$ be the stationary localized solution given in  Theorem 4.1 in \cite{sh_cadiot}. Then, $\tilde{u}$ is nonlinearly unstable and possesses exactly $3$ unstable directions,  including one unstable direction in $\mathcal{H}_{cc}$, one in $\mathcal{H}_{cs}$ and one in $\mathcal{H}_{sc}$. 
\end{theorem}

\begin{proof}
    We want to prove that $D\mathbb{F}(\tilde{u})$ has a kernel of dimension $3$, that it has $3$ positive eigenvalues $\nu_1,\nu_2,\nu_3$ and that the rest of its spectrum is negative. First, we fix $\delta_0 \bydef 10^{-2}$,  $\delta \bydef \mu - \delta_0$ and get that $\sigma_\delta = (-\delta_0,\infty)$. Now, we determine a value of $t$ as given in Proposition \ref{prop : minimum of the spectrum full}. Suppose that $u \in L^2$ satisfies 
    \[
    \mathbb{L}u + D\mathbb{G}(\tilde{u}) u = \lambda u.
    \]
    Then, we have 
    \begin{align}\label{eq : min of spectrum}
       \min_{\xi \in \R^2} |l(\xi)-\lambda| \|u\|_2 \leq  \|\mathbb{L}u-\lambda u\|_2 = \|D\mathbb{G}(\tilde{u})u\|_2.
    \end{align}
    Now, notice that 
    \begin{align}\label{eq : estimation diff DG}
        \|D\mathbb{G}(\tilde{u})u-D\mathbb{G}(u_0)u\|_2 &\leq 2|\nu_1| \|\tilde{u}-u_0\|_\infty \|u\|_2 + 3 \nu_2\|\tilde{u}^2-u_0^2\|_\infty \|u\|_2\\
       & \leq 2|\nu_1| \kappa \|\tilde{u}-u_0\|_\mathcal{H} \|u\|_2 + 3 |\nu_2| \kappa \|u_0-\tilde{u}\|_\mathcal{H}\|u_0 + \tilde{u}\|_\infty\|u\|_2
    \end{align}
where $\kappa = \frac{1}{\mu}\|\frac{1}{l}\|_2$ satisfies $\|u\|_\infty \leq \kappa \|u\|_\mathcal{H}$ for all $u \in \mathcal{H}$ and is given in \cite{sh_cadiot}.  Then, we use \eqref{eq : defect with true solution} to obtain that $\|\tilde{u}-u_0\|_\mathcal{H} \leq r_0$. Moreover
\[
\|u_0 + \tilde{u}\|_\infty \leq 2 \|u_0\|_\infty + \|u_0-\tilde{u}\|_\infty \leq 2\|U_0\|_1 + \kappa r_0.
\]
This implies that 
\begin{align*}
    \|D\mathbb{G}(\tilde{u})u\|_2 &\leq \|D\mathbb{G}(u_0)u\|_2 +  2|\nu_1| \kappa r_0\|u\|_2 + 3 |\nu_2|\kappa r_0 (2\|U_0\|_1 + \kappa r_0)\|u\|_2\\
    &\leq \|v_0\|_\infty \|u\|_2 + 2|\nu_1| \kappa r_0\|u\|_2 + 3 |\nu_2|\kappa r_0 (2\|U_0\|_1 + \kappa r_0)\|u\|_2\\
    &\leq \|V_0\|_1 \|u\|_2 + 2|\nu_1| \kappa r_0\|u\|_2 + 3 |\nu_2|\kappa r_0 (2\|U_0\|_1 + \kappa r_0)\|u\|_2
\end{align*}
  where $V_0 \bydef 2\nu_1 U_0 + 3\nu_2 U_0*U_0$ and $v_0 =  2\nu_1 u_0 + 3\nu_2 u_0^2$.
Going back to \eqref{eq : min of spectrum}, we get
\begin{align*}
     \min_{\xi \in \R^2} |l(\xi)-\lambda| \leq \|V_0\|_1 +  2|\nu_1| \kappa r_0 + 3 |\nu_2|\kappa r_0 (2\|U_0\|_1 + \kappa r_0)
\end{align*}
Let $\lambda_{\max}$ be defined as 
\begin{equation}
    \lambda_{\max} \bydef  \|V_0\|_1 +  2|\nu_1| \kappa r_0 + 3 |\nu_2|\kappa r_0 (2\|U_0\|_1 + \kappa r_0) - \mu,
\end{equation}
then the above implies that if $t < -\lambda_{\max}$, then $D\mathbb{F}(\tilde{u}) +tI$ is invertible. Similarly, one can easily prove that for the same choice of $t$ we also have that $DF(U_0) +t I$ is invertible. Consequently, we fix some $t < -\lambda_{\max}$. In particular, the above reasoning provides that 
    \[
    \text{Eig}(D\mathbb{F}(\tilde{u}))\cap \sigma_\delta \subset (-\delta_0, \lambda_{max}].
    \]
    Consequently, we choose our Jordan domain $\mathcal{J}$ (which is simply a closed interval in this case) as 
    \[
    \mathcal{J} \bydef [-\delta_0,\lambda_{\max}].
    \]

    Now we provide details on the computations of the bounds of Lemma \ref{lem : link of spectrum unbounded to U0} and Lemma \ref{lem : link of spectrum Fourier coeff in q}. For this purpose, we heavily rely on \cite{unbounded_domain_cadiot} and \cite{sh_cadiot}.  Given $\lambda \in \mathcal{J}$, Section 3.5 in \cite{sh_cadiot} provides explicit formulas for $\mathcal{Z}_{u,1}$ and $\mathcal{Z}_{u,2}$ if we can compute $a_\lambda$ and $C_\lambda$ such that
    \begin{align}\label{eq : inequality for Clambda}
        |f_\lambda(x)| \leq C_\lambda e^{-a_\lambda |x|_1}
    \end{align}
    for all $x \in \R^2$ where $f_\lambda \bydef \mathcal{F}^{-1}\left(\frac{1}{l-\lambda}\right)$. Using equation (40) in \cite{sh_cadiot}, we have 
    \begin{align*}
        f_\lambda(x) = \frac{1}{2i\sqrt{\mu+\lambda}}\left(K_0(b_\lambda|x|_2) - K_0(\overline{b_\lambda}|x|_2)\right)
    \end{align*}
    where $b_\lambda \bydef \sqrt{2}a_\lambda - i \frac{\sqrt{\mu+\lambda}}{2\sqrt{2}a_\lambda}$ and $a_\lambda \bydef \frac{\sqrt{-1+\sqrt{1+\mu+\lambda}}}{2}$. In particular, using Section 3.5.1 in \cite{sh_cadiot}, we prove that, given $a_\lambda = \frac{\sqrt{-1+\sqrt{1+\mu+\lambda}}}{2}$,  $C_\lambda = \frac{1.335}{\sqrt{\mu+\lambda}}$ satisfies \eqref{eq : inequality for Clambda} for all $\lambda \in \mathcal{J}$. This allows us to compute $\mathcal{Z}_{u,1}$ and $\mathcal{Z}_{u,2}$ thanks to the use of rigorous numerics (cf. \cite{julia_cadiot} for the numerical details). Moreover, Section 3.5 from \cite{sh_cadiot} also provides an explicit formula for the supremum over $q \in [d,\infty)$ of  $\mathcal{Z}^{(q)}_{u,i}$ $i \in \{1,2\}$. In practice,  we  choose $\mathcal{Z}^{(q)}_{u,i} = 2 \mathcal{Z}_{u,i}$, where $\mathcal{Z}_{u,i}$ is computed thanks to Lemma 3.12 in \cite{sh_cadiot}.

    Concerning the bound $\mathcal{Z}_{u,3}$, we need to control the quantity $\|\pi^N (S+ tI)^{-1}P^{-1}(L-\lambda I)\|_2 $ for all $\lambda \in \mathcal{J}$. Since $(S+tI)^{-1}$ is diagonal, we have $\pi^N(S+tI)^{-1} = \pi^{N}(S+tI)^{-1}\pi^{N}$. Moreover, by construction of $P^{-1}$, we have $\pi^{N}P^{-1} = \pi^{N}(P^N)^{-1}\pi^{N}$. Finally, since $L-\lambda I$ is also diagonal, we get
    \begin{align*}
        \|\pi^N (S+ tI)^{-1}P^{-1}(L-\lambda I)\|_2 = \|\pi^N (S+ tI)^{-1}(P^{N})^{-1}(L-\lambda I)\pi^{N}\|_2,
    \end{align*}
    which is essentially a matrix norm. Now, we have the following
    \footnotesize{
    \begin{align}\label{eq : estimation of Zu3 matrix}
         \|\pi^N (S+ tI)^{-1}(P^{N})^{-1}(L-\lambda I)\pi^{N}\|_2 \leq \|\pi^N (S+ tI)^{-1}(P^{N})^{-1}(L + \delta_0 I)\pi^{N}\|_2 \|(L+\delta_0 I)^{-1}(L-\lambda I)\|_2,
    \end{align}
    }\normalsize
    where the right hand side is explicitly computable for every $\lambda$ thanks to rigorous numerics. Combined with Lemma 3.12 in \cite{sh_cadiot}, \eqref{eq : estimation of Zu3 matrix} allows to compute $\mathcal{Z}_{u,3}$ and we use that $\mathcal{Z}_{u,3}^{(q)} \leq 2 \mathcal{Z}_{u,3}$ for all $q \in [d,\infty)$ in this case.   

    We now focus on the bounds  $Z_{1,i}$ $(i \in \{1,2,3,4\})$. First, note that $Z_{1,3}$ involves a matrix norm, which is handled via rigorous numerics. For the bound $Z_{1,4}$, we use the properties of the adjoint to get
    \begin{align*}
        \|\pi^{N}(S+tI)^{-1}R\pi_N\|_2^2 = \|\pi^{N}(S+tI)^{-1}R\pi_NR^*(S^*+tI)^{-1}\pi^N\|_2.
    \end{align*}
    Now, notice that $\pi^NR\pi_N = (P^N)^{-1}DG(U_0)\pi_N$ by construction in \eqref{eq : pseudo diagonalization DF}. This implies that 
     \begin{align*}
        \|\pi^{N}(S+tI)^{-1}R\pi_N\|_2^2 = \|\pi^{N}(S+tI)^{-1}(P^N)^{-1}DG(U_0)\pi_N DG(U_0)^*((P^N)^*)^{-1}(S^*+tI)^{-1}\pi^N\|_2.
    \end{align*}
    Again, the right hand side is a matrix norm, which is tackled on the computer.  Concerning $Z_{1,1}$ and $Z_{1,2}$, we have 
    \begin{align*}
          \sup_{\lambda \in \mathcal{J}}\|\pi_N (L -\lambda I)^{-1}R \pi^N\|_2 &\leq \|\pi_N (L +\delta_0 I)^{-1}R \pi^N\|_2  \\
          \sup_{\lambda \in \mathcal{J}}\|\pi_N (L -\lambda I)^{-1}R\pi_N\|_2 &\leq \|\pi_N (L +\delta_0 I)^{-1}R\pi_N\|_2,
    \end{align*}
    since $\|(L-\lambda I)^{-1}(L+\delta_0 I)\|_2 \leq 1$ for all $\lambda \in \mathcal{J}$. The computation of the right hand side can now be computed using some standard computer-assisted analysis, as presented in Section \ref{sec : control spectrum DF fourier} or as in Section 4.3 in \cite{unbounded_domain_cadiot}, and is exposed in \cite{julia_cadiot}. 
    
    Finally, for the computation of $\mathcal{C}_1$ and $\mathcal{C}_2$, it remains to estimate $\left\| \left(\mathbb{L} - \lambda I\right)^{-1} \left(D\mathbb{G}(u_0)-D\mathbb{G}(\tilde{u})\right)\right\|_2$ for all $\lambda \in \mathcal{J}$. In particular, since $l(\xi) \leq -\mu$ for all $\xi \in \R^2$, we have 
    \begin{align*}
         \left\| \left(\mathbb{L} - \lambda I\right)^{-1} \left(D\mathbb{G}(u_0)-D\mathbb{G}(\tilde{u})\right)\right\|_2 \leq \frac{1}{\mu + \lambda} \|\left(D\mathbb{G}(u_0)-D\mathbb{G}(\tilde{u})\right)\|_2
    \end{align*}
    for all $\lambda \in \mathcal{J}$. Moreover, we compute $\|D\mathbb{G}(u_0)-D\mathbb{G}(\tilde{u})\|_2$ using \eqref{eq : estimation diff DG}.

    Now that we presented the computation of the bounds of Lemma \ref{lem : link of spectrum unbounded to U0} and \ref{lem : link of spectrum Fourier coeff in q}, we implement the formulas in \cite{julia_cadiot} using the arithmetic on intervals. In particular, we use the derived bounds for each restriction $D\mathbb{F}_{ij}(\tilde{u}) : \mathcal{H}_{ij} \to L^2_{ij}$  separately, where $i,j \in \{c,s\}$, and we apply Theorem \ref{th : gershgorin unbounded}. 

    We apply a first time Theorem \ref{th : gershgorin unbounded} in \cite{julia_cadiot} and obtain Gershgorin disks which are not disjoint. This implies that we cannot precisely locate each eigenvalue. However, considering the union of all disks, we obtain that the eigenvalues must be smaller than $0.08$. This allows us to choose a smaller value for $t$ and we choose $t = -0.11$. 

    We apply Theorem \ref{th : gershgorin unbounded} a second time (with $t= -0.11$) and we get the following.
    For $i = j= c$, we obtain that $D\mathbb{F}_{cc}(\tilde{u})$ possesses a positive eigenvalue $\nu_1 \in [0.047,0.054]$ and that the rest of its spectrum is negative. For $i = s$ and $j = c$, we obtain that $D\mathbb{F}_{sc}(\tilde{u})$ possesses a positive eigenvalue $\nu_2 \in [0.033,0.041]$, an eigenvalue $\tilde{\nu}_2 \in [-0.005,0.005]$ and the rest of its spectrum is negative. For $i=c$ and $j=s$, we obtain that $D\mathbb{F}_{cs}(\tilde{u})$ possesses a positive eigenvalue $\nu_3 \in [0.033,0.041]$, an eigenvalue $\tilde{\nu}_3 \in [-0.005,0.005]$ and the rest of its spectrum is negative. Finally, $D\mathbb{F}_{ss}(\tilde{u})$ has an eigenvalue  $\tilde{\nu}_4 \in [-0.0036 ,0.0036]$ and the rest of its spectrum is negative.

     Using the above, we have found that three eigenvalues $\tilde{\nu}_2,\tilde{\nu}_3,\tilde{\nu}_4$ are contained in intervals containing zero. Moreover, these intervals are disjoint from the rest of the spectrum. Then, we verify rigorously that $\tilde{u}$ is not radially symmetric using rigorous numerics, by proving that $|\tilde{u}(x) - \tilde{u}(y)| >0$ for some $x,y \in \R^2$ such that $|x|_2 = |y|_2$. Using Lemma \ref{lem : radially symmetric}, this implies that the kernel of $D\mathbb{F}(\tu)$ is at least of dimension 3. We conclude that $\tilde{\nu}_2 =\tilde{\nu}_3 =\tilde{\nu}_4 =0$ and  correspond to the trivial kernel of $D\mathbb{F}(\tilde{u})$ generated by $SE(2)$. 
     
The above provides the linear unstability of $\tilde{u}$.
     For the nonlinear unstability, we conclude the proof  using Section 5.1 of  \cite{henry_semilinear}.
\end{proof}

\subsubsection{The stable hexagonal pattern}\label{ssec : stable hexagonal pattern}

In this section, we investigate the stability of the localized hexagonal pattern $\tilde{u}$, whose existence is established in Theorem 4.2 of \cite{sh_cadiot}. We recall that $\mathbb{F}(\tilde{u}) = 0$ with $\mu = 0.32$, $\nu_1 = -1.6$ and $\nu_2 = 1$. Moreover, we denote $u_0 \in \mathcal{H}_{cc}$ the associated approximate solution (represented in Figure \ref{fig : hexagone pattern}) satisfying \eqref{eq : approximate solution} and \eqref{eq : defect with true solution} with $r_0 = 8.21 \times 10^{-6}$ in that case (cf. Theorem 4.2 in \cite{sh_cadiot}).
\begin{figure}[H]
\centering
 \begin{minipage}{.55\linewidth}
  \centering\epsfig{figure=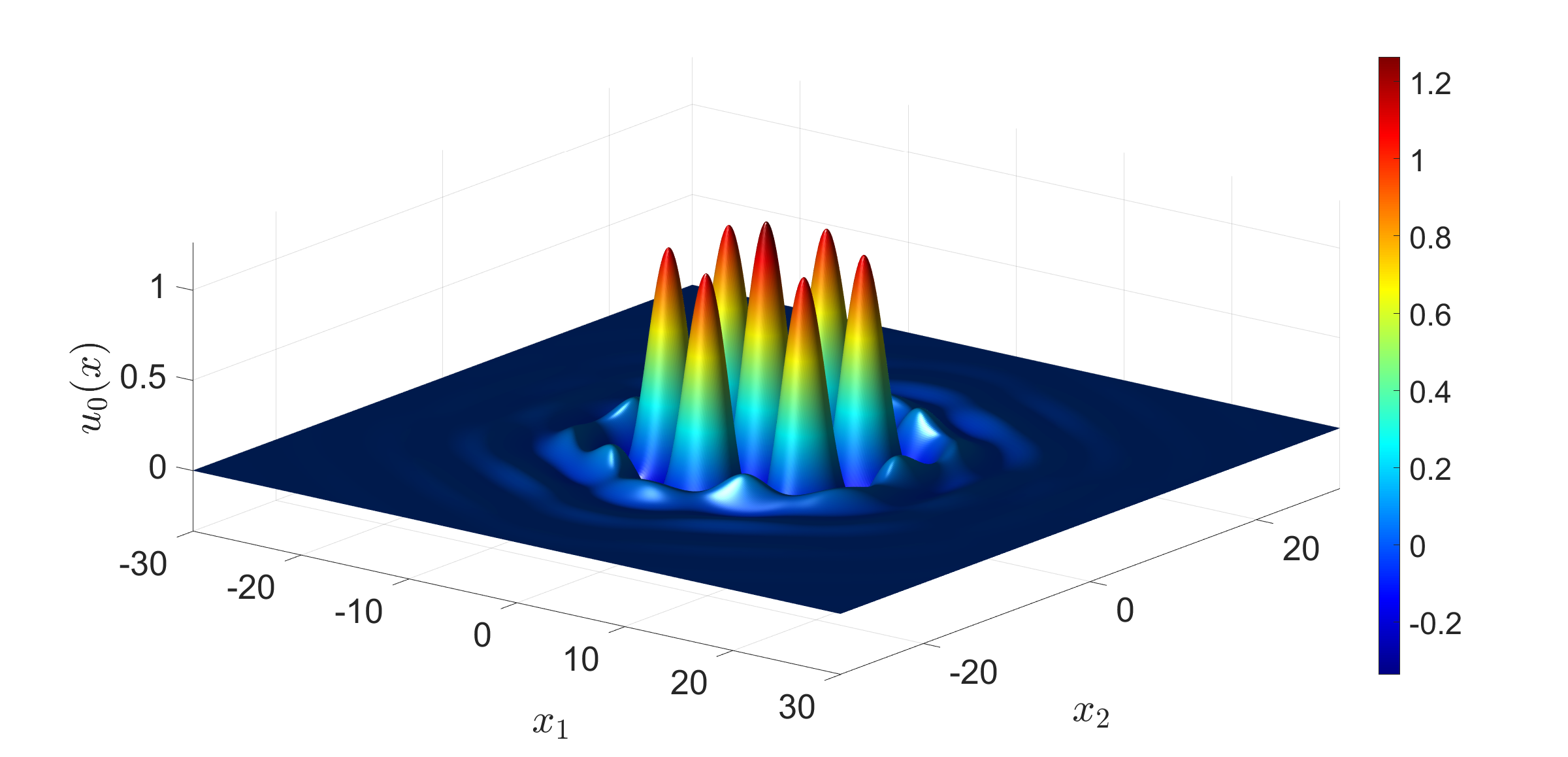,width=\linewidth}
 \end{minipage}%
 \begin{minipage}{.55\linewidth}
  \centering\epsfig{figure=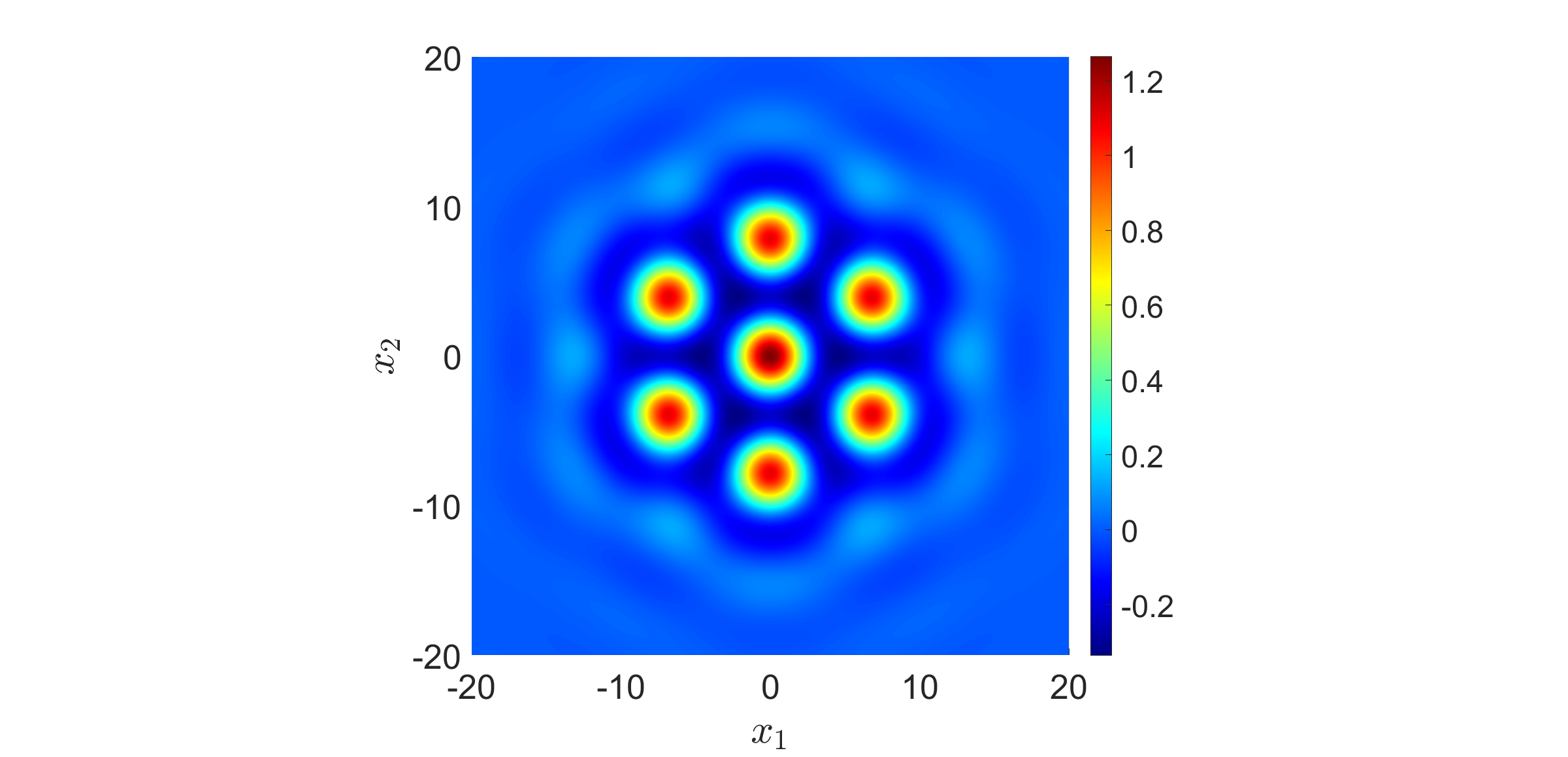,width=\linewidth}
 \end{minipage} 
 \caption{Side and upper views of the approximate solution $u_0$ of the hexagonal pattern established in Theorem 4.2 in \cite{sh_cadiot}} \label{fig : hexagone pattern}
 \end{figure}
In this case, we want to show that $\tilde{u}$ is stable modulo trivial invariances. We first define such a notion of stability. Let $\mathcal{S}(\tilde{u})$ be the set defined as 
\begin{align}\label{eq : invariances of tilde u}
   \mathcal{S}(\tilde{u}) =  \{u \in \mathcal{H}, ~ u = g \tilde{u} \text{ for some } g \in SE(2)\}.
\end{align}
 In other terms, $\mathcal{S}(\tilde{u})$ is the set of all possible translations  and rotations of $\tilde{u}$ on the plane.  Then, we define the stability modulo trivial invariances as follows.
\begin{definition}\label{def : stable modulo trivial invariances}
    We say that $\tilde{u}$ is \emph{stable modulo trivial invariances} if for all $\epsilon >0$, there exists $\alpha>0$ such that if $\|\tilde{u}-v_0\|_\mathcal{H} < \alpha,$ then there exists a unique solution $v= v(t)$ to \eqref{eq : swift original} such that $v(0) = v_0$ and defined for all $0 \leq t < \infty$ such that $\sup_{t \in [0,\infty)} \|v(t)-w\|_{2} \leq \epsilon$, where $w \in \mathcal{S}(\tilde{u}).$
\end{definition}
Note that the definition above correspond to a nonlinear notion of stability. In order to prove that $\tilde{u}$ is stable modulo trivial invariances, we show that the kernel of $D\mathbb{F}(\tilde{u})$ is of dimension $3$, containing exactly the trivial invariances of translation and rotation.
\begin{theorem}\label{th : stable hexagon solution}
    Let $\tilde{u}$ be the stationary localized solution  of \eqref{eq : swift original} given in Theorem 4.2 in \cite{sh_cadiot}. Then, $\tilde{u}$ is  stable modulo trivial invariances.
\end{theorem}
\begin{proof}
    Following very similar estimations and  strategy as for the proof of Theorem \ref{th : unstable square solution}, we prove in \cite{julia_cadiot} that the kernel of $D\mathbb{F}(\tilde{u})$ is of dimension $3$ and that the rest of its spectrum is negative.

    In order to conclude about stability modulo trivial invariances, we use \cite{sandstede_1997_spiral}. Indeed, our goal is to apply Theorem 1 in \cite{sandstede_1997_spiral}. To do so, we prove that the Hypotheses 1-3 in Section 2 are satisfied. First, notice that the operator $\mathbb{L} : H^4(\R^2) \to L^2(\R)$ generates a semi-flow $\phi_t$ for \eqref{eq : swift original}, since \eqref{eq : swift original} is semi-linear and $\mathbb{L}$ has a strictly negative spectrum, uniformly bounded from above (see \cite{henry_semilinear}).  Consequently, since $\tilde{u}$ is a steady state of \eqref{eq : swift original}, we have that Hypothesis 1 is readily satisfied. Concerning Hypothesis 2, we verify it using Lemma 5.4 from \cite{sandstede_1997_spiral} since Hypothesis 4 is satisfied. Finally, we prove Hypothesis 3 using Section 4 of \cite{sandstede_1997_spiral} and the fact that $SE(2)$ acts smoothly on $\tilde{u}$ (since $\tilde{u}$ is smooth and localized).
\end{proof}

\begin{remark}
    Note that \cite{henry_semilinear} and \cite{sandstede_1997_spiral} provide a stronger notion of stability for $\tilde{u}$. Indeed, we have that $\tilde{u}$ is asymptotically stable in the sense that there exists $C>0$, $\beta >0$ and $t_0$ such that $\|v(t) - w\|_2 \leq C e^{-\beta t}$ for all $t \geq t_0$, where $v$ and $w$ are the functions of Definition \ref{def : stable modulo trivial invariances}.
\end{remark}

\subsection{The capillary-gravity Whitham equation}\label{ssec : application whitham}

We turn our interest to the capillary-gravity Whitham equation, which is a nonlocal equation reading as follows
\begin{equation}\label{eq : original Whitham}
   u_t +  \partial_x{\mathbb{M}_T}u + \frac{1}{2} u \partial_xu =0,  ~ u = u(x,t), ~ x \in \R,
\end{equation}
where $\mathbb{M}_T$ is a Fourier multiplier operator defined via its symbol 
\begin{equation}\label{def : Kernel whitham}
        \mathcal{F}(\mathbb{M}_Tu)(\xi) \bydef m_T(2\pi\xi) \hat{u}(\xi) \bydef \sqrt{\frac{\tanh(2\pi\xi)(1+T(2\pi\xi)^2)}{2\pi\xi}}\hat{u}(\xi)
\end{equation}
for all $\xi \in \mathbb{R}$. The quantity $T > 0$ is the Bond number accounting for the capillary effects (also known as surface tension). In particular, we are interested in localized traveling wave solutions, also known as \emph{solitary waves}, to \eqref{eq : original Whitham}. Using the ansatz $X = x-ct$, we look for $u : \R \to \R$ satisfying
    \begin{align}\label{eq : whitham stationary}
   \mathbb{F}(u) \bydef {\mathbb{M}_T}u -cu + u^2 =0
\end{align}
where $c \in \mathbb{R}$ and $u(x) \to 0$ as $|x| \to \infty.$  Using the notations of Section \ref{sec : presentation of the problem}, we have 
\begin{align*}
    \mathbb{L} = \mathbb{M}_T - c I, ~~ l(\xi) = m_T(2\pi\xi) - c  ~~  \text{ and } ~~  \mathbb{G}(u) = u^2.
\end{align*}
In this subsection, we fix $c =0.8$ and $T = 0.5$.
Our goal is to study the spectrum of $D\mathbb{F}(\tilde{u})$, where $\tilde{u}$ is the solution for which the existence was constructively established in Theorem 4.8 of \cite{cadiot2024constructiveproofsexistencestability} (denoted $\tilde{u}_3$) in a vicinity of an approximate solutio $u_0$ represented in Figure \ref{fig : whitham}. In particular, we have that 
\begin{align*}
    \|(I - T\Delta)\mathbb{L}(\tilde{u}-u_0)\|_2 \leq r_0 = 8.7\times 10^{-9}.
\end{align*}
\begin{figure}[H]
\centering
 \begin{minipage}{.7\linewidth}
\centering\epsfig{figure=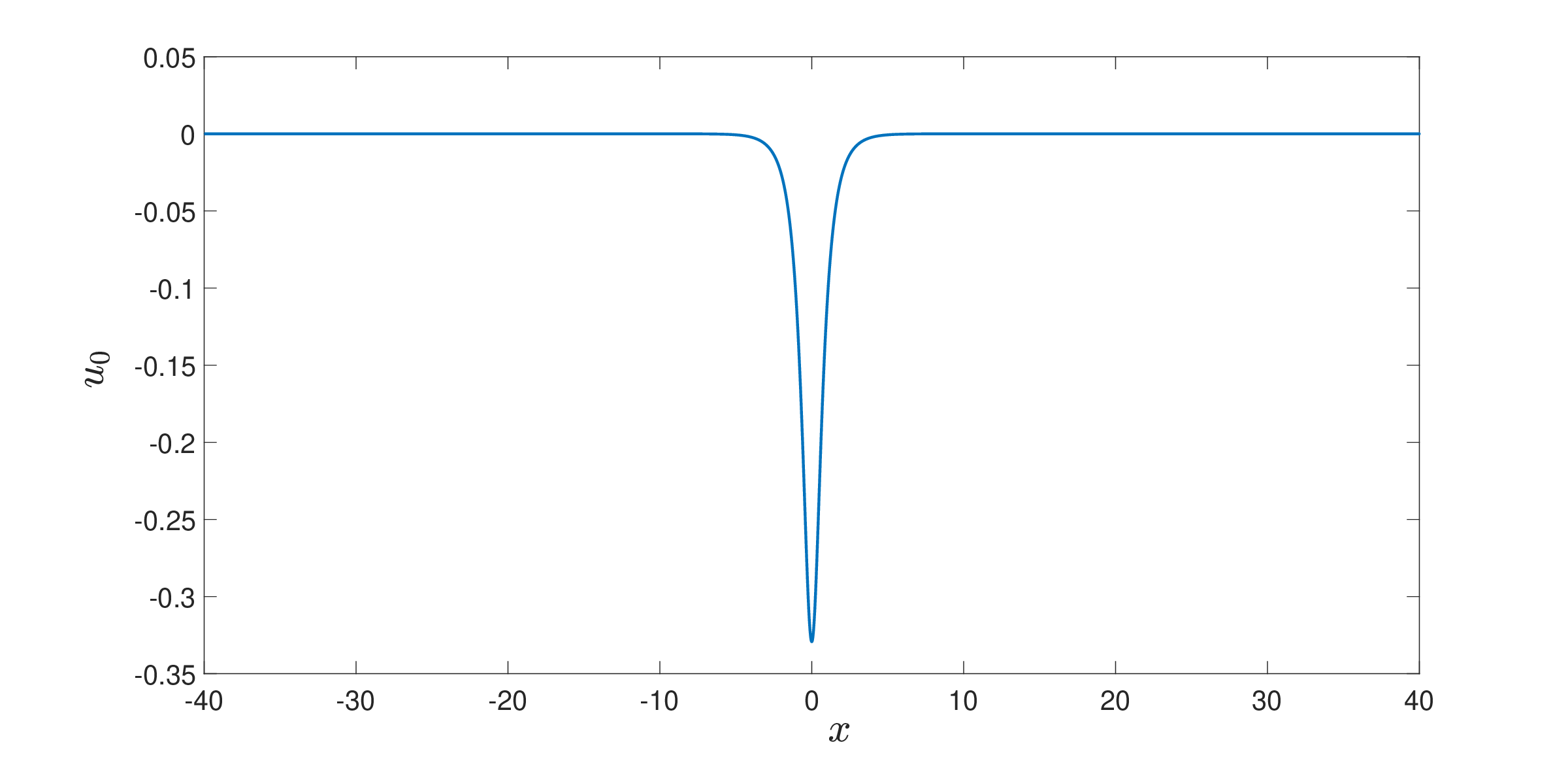,width=\linewidth}
 \end{minipage}%
 \caption{Approximate solution $u_0$ of the solitary wave  $\tilde{u}_3$ established in Theorem 4.8 in \cite{sh_cadiot}} \label{fig : whitham}
 \end{figure}
Note that in \cite{cadiot2024constructiveproofsexistencestability}, the spectral (linear) stability of $\tilde{u}$ has already been established combining computer-assisted proofs of isolated eigenvalues and non-existence of eigenvalues well-chosen regions of the real line. Our goal is to demonstrate that the framework of this manuscript allows to obtain the spectral stability much more efficiently.

Similarly as the Swift-Hohenberg case, the existence of $\tilde{u}$ is established in \cite{cadiot2024constructiveproofsexistencestability} in the even subspace of $\mathcal{H}$. Consequently, we denote 
\begin{align*}
    \mathcal{H}_c &\bydef \left\{u \in \mathcal{H}, ~ u(x) = u(-x) \text{ for all } x \in \R\right\}\\
    \mathcal{H}_s &\bydef \left\{u \in \mathcal{H}, ~ u(x) = -u(-x) \text{ for all } x \in \R\right\}.
\end{align*}
In terms of Fourier series, $\mathcal{H}_c$ will lead to cosine series and $\mathcal{H}_s$ to sine series. Now, since $\tilde{u}, u_0 \in \mathcal{H}_c$, we can consider $D\mathbb{F}_c(\tilde{u}) : \mathcal{H}_c \to L^2_c$ and $D\mathbb{F}_s(\tilde{u}) : \mathcal{H}_s \to L^2_s$, the restrictions of $D\mathbb{F}(\tilde{u})$ to $\mathcal{H}_c \to L^2_c$ and $\mathcal{H}_s \to L^2_s$ respectively. Moreover, similarly as for Section \ref{ssec : application SH}, we have that 
\begin{align*}
    \sigma(D\mathbb{F}(\tilde{u})) = \sigma(D\mathbb{F}_c(\tilde{u}))\cup \sigma(D\mathbb{F}_s(\tilde{u})).
\end{align*}
Using such a decomposition, we apply Theorem \ref{th : gershgorin unbounded} and obtain the following.
\begin{lemma}\label{lem : enclosures whitham}
    Let $\tilde{u}$ be the solitary wave given in Theorem 4.8 in \cite{cadiot2024constructiveproofsexistencestability} corresponding to $T=0.5$ and $c=0.8$. Then,  $\sigma_{ess}(D\mathbb{F}(\tilde{u})) = [0.2, \infty)$. Moreover,   $D\mathbb{F}(\tilde{u})$ possesses exactly three eigenvalues in  $(-\infty, 0.16]$, denoted $\nu_1, \nu_2, \nu_3$, which satisfy the following 
    \begin{equation}\label{eq : eigenvalues nui in whitham}
        \nu_1 \in [0.2691, 0.2704], ~~ 
        \nu_2 =0 ~~ \text{ and } ~~
        \nu_3 \in [-0.1294, -0.1268].
     \end{equation}
    Finally, $\tilde{u}$ is spectrally stable.
\end{lemma}

\begin{proof}
First, notice that $l(\xi) \geq l(0) = 1-c = 0.2$ for all $\xi \in \R$. This implies that  $\sigma_{ess}(D\mathbb{F}(\tilde{u})) = [0.2, \infty)$ (cf. Lemma \ref{lem : separation of the spectrum sigma 0}). Using that $D\mathbb{F}(\tilde{u}) : L^2 \to L^2$ is self-adjoint, we have that $\sigma(D\mathbb{F}(\tilde{u})) \subset \R$. In particular,  we fix $\delta = 0.4$ and consider the restriction of $\sigma_\delta$ to $\R$, which we still denote $\sigma_\delta$ for convenience. In particular, $\sigma_\delta = (-\infty,0.16)$.

    In order to compute values for the sequence $(\epsilon_n)$ from Theorem \ref{th : gershgorin unbounded}, we use the results of Section 5 from \cite{cadiot2024constructiveproofsexistencestability}. In particular, these results provide explicit formulas for the bounds of Lemma \ref{lem : link of spectrum unbounded to U0} and \ref{lem : link of spectrum Fourier coeff in q}. For instance, a value for $t$ is given in Lemma 5.2 in \cite{cadiot2024constructiveproofsexistencestability}. Following the reasoning of the proof of Theorem \ref{th : unstable square solution}, we implement such formulas in \cite{julia_cadiot} and apply Theorem \ref{th : gershgorin unbounded}.

    Applying Theorem \ref{th : gershgorin unbounded} to the even restriction $D\mathbb{F}_c(\tilde{u})$, we obtain the enclosures for $\nu_1$ and $\nu_3$ given in \eqref{eq : eigenvalues nui in whitham}. In particular, the disks containing $\nu_1$ and $\nu_3$ do not intersect the rest of the disks given by Theorem \ref{th : gershgorin unbounded}.  Then, applying similarly Theorem \ref{th : gershgorin unbounded} to the odd restriction $D\mathbb{F}_c(\tilde{u})$, we obtain that $\nu_2 \in [-0.0011,0.0011]$. However, we know that $D\mathbb{F}(\tilde{u})$ has at least a 1D kernel coming from the translation invariance of solutions. This implies that $\nu_2 = 0$. We conclude the spectral stability proof using the negativity of the Vakhitov–Kolokolov quantity established in \cite{cadiot2024constructiveproofsexistencestability}.
\end{proof}

\begin{remark}
Using the previous lemma, we are able to control the spectrum of $D\mathbb{F}(\tilde{u})$ on \\ $(-\infty,0.16]\cup[0.2,\infty)$. 
   Numerically, we observe that $D\mathbb{F}(\tilde{u})$ possesses an eigenvalue around $0.19$, which we were not able to enclose. In order to enclose this eigenvalue, we would need to choose $\delta < 0.1$ in the  proof of Lemma \ref{lem : enclosures whitham}. In this case, our estimations lead to bounds $\mathcal{Z}_{u,i}$ in Lemma \ref{lem : link of spectrum unbounded to U0} which are too big to obtain non-intersecting Gershgorin disks in Theorem \ref{th : gershgorin unbounded}. In other words, the eigenvalues which are close to the essential spectrum lead to less localized eigenvectors. In turn, such eigenvectors are more difficult to approximate using a Fourier series on $\om$.
\end{remark}

\subsection{The planar Gray-Scott model}\label{sec : application GS}

In this section, we investigate the stability of localized stationary patterns in the planar Gray-Scott model :
\begin{equation}\label{eq : gray_scott cov}
  \partial_t \mathbf{u} = \mathbf{D}_1 \Delta \mathbf{u} + \mathbf{D}_2 \mathbf{u} + \mathbb{G}(\mathbf{u}) = \mathbb{F}(\mathbf{u}) ; ~~ \mathbf{u} = (u_1,u_2) : \R^2 \to \R^2
\end{equation}
where $\mathbf{D}_1$ and $\mathbf{D}_2$ are matrices given as  
\begin{align*}
    \mathbf{D}_1 \bydef  \begin{pmatrix}
        \lambda_1 & 0\\
        0 & 1
    \end{pmatrix} ~~ \text{ and } ~~ \mathbf{D}_2 \bydef \begin{pmatrix}
        -1 & 0\\
        \lambda_1\lambda_2-1 & -\lambda_2
    \end{pmatrix}, 
\end{align*}
and $\lambda_1, \lambda_2 >0$ are given parameters. Moreover, the nonlinear term $\mathbb{G}$ is defined as follows
\begin{align*}
    \mathbb{G}(\mathbf{u}) \bydef  \begin{pmatrix}
         (u_2 + 1 - \lambda_1 u_1)u_1^2 \\ 0
    \end{pmatrix}.
\end{align*}
Finally, we define $\mathbb{L}$ as 
\begin{equation}\label{eq : def of L gray scott}
    \mathbb{L} \bydef \mathbf{D}_1 \Delta + \mathbf{D}_2 I  = \begin{pmatrix}
    \lambda_1 \Delta - I & 0\\
    (\lambda_1\lambda_2-1) I & \Delta - \lambda_2 I
\end{pmatrix}
\end{equation}
and $l$ its Fourier transform as 
\begin{equation}\label{eq : def of l in gray scott}
    l(\xi) \bydef \begin{pmatrix}
    -\lambda_1 |2\pi\xi|_2^2 -1 & 0\\
     \lambda_1\lambda_2 -1 & -|2\pi\xi|_2^2 - \lambda_2
\end{pmatrix} ~~ \text{ for all } \xi \in \R^2.
\end{equation}

 Until that point, we have focused on scalar equations. However, \cite{cadiot20242dgrayscottequationsconstructive} provides that the bounds derived in Lemmas \ref{lem : link of spectrum unbounded to U0} and \ref{lem : link of spectrum Fourier coeff in q} can be explicitly computed in systems of PDEs satisfying Assumptions 1 and 2 in \cite{cadiot20242dgrayscottequationsconstructive}, which is the case for \eqref{eq : gray_scott cov}. In fact, the estimations derived in the previous sections can be applied to the different components of the system. For simplicity, we still use the scalar notations used in the previous section and naturally generalize them for system of PDEs as in \cite{cadiot20242dgrayscottequationsconstructive}. For instance, we denote $L^2 = L^2(\R^2) \times L^2(\R^2)$ and $\mathcal{H} = \{\mathbf{u} \in L^2, ~ \|\mathbf{u}\|_{\mathcal{H}} = \|\mathbb{L}\mathbf{u}\|_2 < \infty\}$. 
\begin{figure}[H]
\centering
 \begin{minipage}{.5\linewidth}
  \centering\epsfig{figure=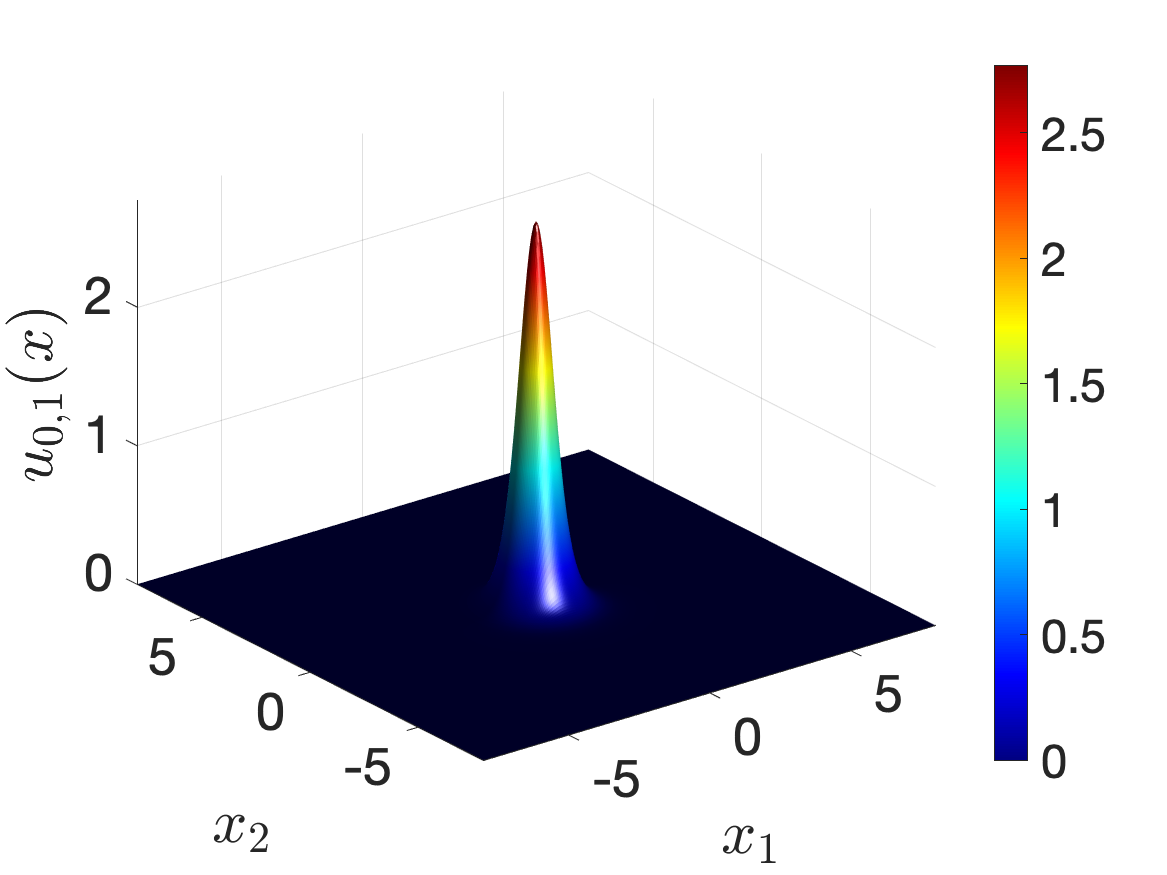,width=\linewidth}
 \end{minipage}%
 \begin{minipage}{.5\linewidth}
  \centering\epsfig{figure=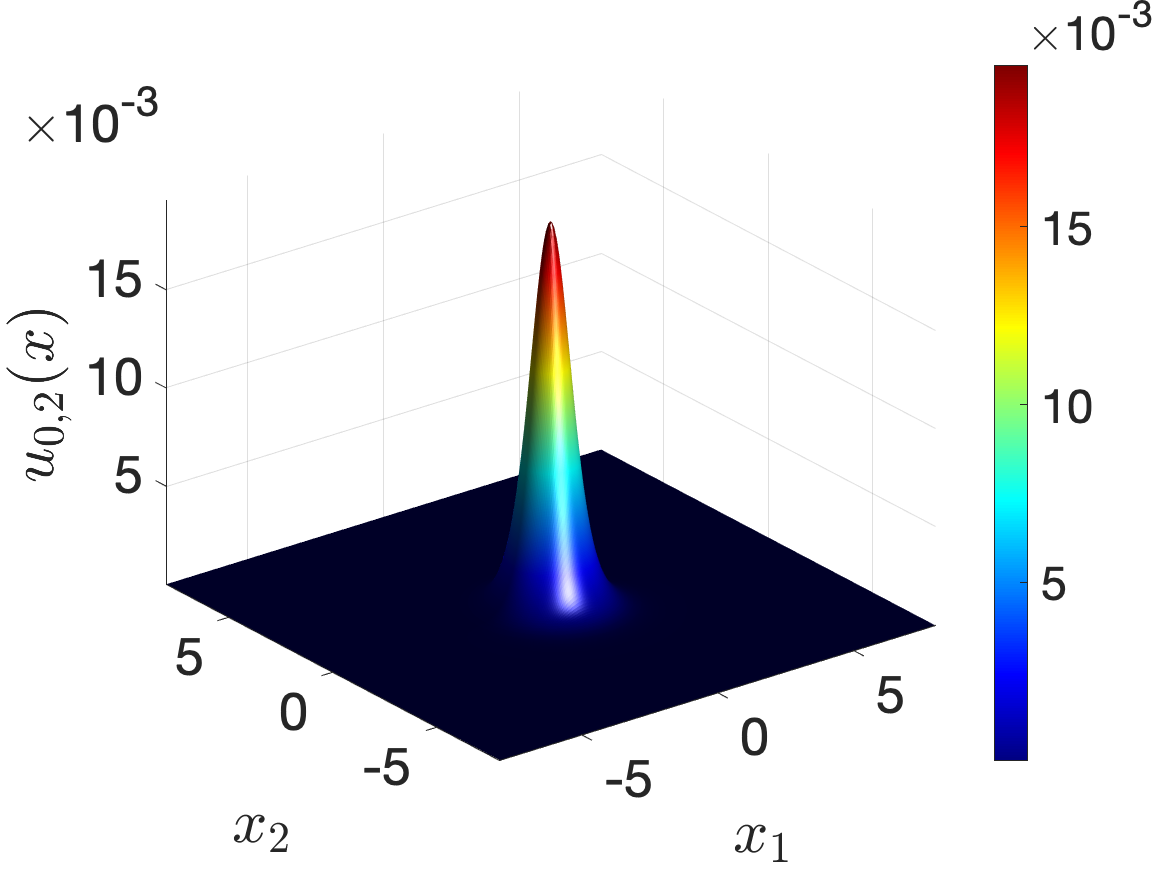,width=\linewidth}
 \end{minipage} 
 \caption{Approximation of the spike pattern established in Theorem 6.2 in \cite{cadiot20242dgrayscottequationsconstructive}. (L) is $u_{0,1}$ and (R) is $u_{0,2}$. }\label{fig : gray scott}
 \end{figure}
In the rest of this section, we will focus on the localized pattern given in Theorem 6.2 of \cite{cadiot20242dgrayscottequationsconstructive}.  In particular, we fix $\lambda_1 = \frac{1}{9}$ and $\lambda_2 = 10$, then Theorem 6.2 in \cite{cadiot20242dgrayscottequationsconstructive}  provides the existence of a localized stationary solution $\tilde{\mathbf{u}}$ to \eqref{eq : gray_scott cov} such that 
\begin{align}\label{eq : u0 for gray scott}
    \|\mathbf{\tilde{u}} - \mathbf{u}_0\|_{\mathcal{H}} \leq r_0 = 6\times 10^{-6}.
\end{align}
$\mathbf{u}_0 = (u_{0,1}, u_{0,2}) \in \mathcal{H}$  is the approximate solution given in \cite{cadiot20242dgrayscottequationsconstructive}, and represented in Figure \ref{fig : gray scott}. $\mathbf{u}_0$ is associated to its Fourier coefficients representation $\mathbf{U}_0.$
Moreover, $\tilde{\mathbf{u}}$ is a $D_4$-symmetric function, that is it is invariant under the symmetries of the square (corresponding to the dihedral group $D_4$). Since the dihedral group $D_2$ is a subgroup of $D_4$, we obtain that $\tilde{\mathbf{u}} \in \mathcal{H}_{cc}$ (since $\mathcal{H}_{cc}$ corresponds to $D_2$-symmetric functions in $\mathcal{H}$), where $\mathcal{H}_{cc}$ is given as  in Section \ref{ssec : application SH}. In particular, similarly as what was achieved in Section \ref{ssec : application SH}, we  have the following disjoint decomposition  
\begin{equation}\label{eq : identity symmetry spectrum}
    \sigma\left(D\mathbb{F}(\tilde{\mathbf{u}})\right) =  \sigma\left(D\mathbb{F}_{cc}(\tilde{\mathbf{u}})\right) \cup  \sigma\left(D\mathbb{F}_{cs}(\tilde{\mathbf{u}})\right) \cup  \sigma\left(D\mathbb{F}_{sc}(\tilde{\mathbf{u}})\right) \cup  \sigma\left(D\mathbb{F}_{ss}(\tilde{\mathbf{u}})\right),
\end{equation}
where $D\mathbb{F}_{ij}(\tilde{\mathbf{u}})$ is the restriction of $D\mathbb{F}(\tilde{\mathbf{u}})$ to $\mathcal{H}_{ij}$ for all $i,j \in \{c,s\}$. Moreover, since $\tilde{\mathbf{u}}$ is $D_4$-symmetric, we have that $\tilde{\mathbf{u}}(x_1,x_2) = \tilde{\mathbf{u}}(x_2,x_1)$ for all $(x_1,x_2) \in \R^2$. From this property, we obtain that 
\[
\sigma\left(D\mathbb{F}_{cs}(\tilde{\mathbf{u}})\right) =  \sigma\left(D\mathbb{F}_{sc}(\tilde{\mathbf{u}})\right),
\]
hence we simply need to compute $\sigma\left(D\mathbb{F}_{cc}(\tilde{\mathbf{u}})\right), \sigma\left(D\mathbb{F}_{cs}(\tilde{\mathbf{u}})\right)$ and $\sigma\left(D\mathbb{F}_{ss}(\tilde{\mathbf{u}})\right)$. In order to compute the estimations presented in Section \ref{sec : control spectrum DF unbounded}, we heavily rely on the analysis of Section 4 in \cite{cadiot20242dgrayscottequationsconstructive}, which provides explicit formulas for the application of Theorem \ref{th : gershgorin unbounded}. 

Moreover, we demonstrate how Theorem \ref{th : gershgorin unbounded} can be used to derive additional properties on the solution. In particular, using Lemma \ref{lem : radially symmetric}, we obtain information on the symmetries of $\tmu$.
\begin{theorem}\label{th : unstable gray scott}
    Let $\tilde{\mathbf{u}}$ be the stationary localized solution of \eqref{eq : gray_scott cov} given in  Theorem 6.2 in \cite{cadiot20242dgrayscottequationsconstructive}.  Then $\tilde{\mathbf{u}}$ is nonlinearly unstable and possesses exactly 1 unstable direction which is in $\mathcal{H}_{cc}$. Moreover $\tmu$ is radially symmetric.
\end{theorem}

\begin{proof}
    The proof is very similar to the proof of Theorem \ref{th : unstable square solution}. In particular, we rely on \cite{cadiot20242dgrayscottequationsconstructive} in order to compute the sequence $(\epsilon_n)_{n \in \mathbb{Z}^2}$ in Theorem \ref{th : gershgorin unbounded}.

    First, using \eqref{eq : def of l in gray scott}, we obtain that 
    \begin{align}
        \sigma_{ess}(D\mathbb{F}(\tmu)) = (-\infty, ~\max\{-1,-\lambda_2\}] = (-\infty,-1]
    \end{align}
    since $\lambda_2 = 10$. Now, we find a value for $t$ such that $D\mathbb{F}(\tmu) + tI$ and $DF(\mathbf{U}_0) + tI$ are invertible as in Proposition \ref{prop : minimum of the spectrum full}. Let $\lambda \in \sigma_0$ be an eigenvalue of $D\mathbb{F}(\tmu)$ associated to an eigenvector $\mathbf{u}= (u_1,u_2)$ such that $\|\mathbf{u}\|_2=1$. Then, taking the second equation of $D\mathbf{F}(\tmu)\mathbf{u}-\lambda \mathbf{u} = 0$, we have that
    \begin{align*}
        (\lambda_1\lambda_2-1)u_1 + (\Delta - \lambda_2 - \lambda)u_2 = 0
    \end{align*}
This implies that 
\begin{align}\label{eq : proof GS th 1}
    |\lambda_2+\lambda| \|u_2\|_2 \leq \|(\Delta - \lambda_2 - \lambda)u_2\|_2 = \| (\lambda_1\lambda_2-1)u_1\|_2 \leq |\lambda_1\lambda_2 -1| \|u_1\|_2.
\end{align}
In particular, using the above, we have that $\|u_1\|_2 >0$.
Now, using the first equation of $D\mathbf{F}(\tmu)\mathbf{u}-\lambda \mathbf{u} = 0$, we get
\begin{align*}
    (\Delta-1-\lambda)u_1 + \tilde{v}_1u_1 + \tilde{v}_2u_2 = 0
\end{align*}
where $\tilde{v}_1 = 2\tilde{u}_{1}\tilde{u}_2 + 2\tilde{u}_1 - 3\lambda_1\tilde{u}_1^2$ and $\tilde{v}_2 = \tilde{u}_1^2$.
Using \eqref{eq : proof GS th 1} and the fact that $\lambda \in \sigma_0$, this implies that 
\begin{align}
    |1+\lambda| \|u_1\|_2 \leq \| \tilde{v}_1\|_\infty \|u_1\|_2 + \| \tilde{v}_2\|_\infty \|u_2\|_2 \leq \left( \| \tilde{v}_1\|_\infty + \frac{|\lambda_1\lambda_2-1|}{|\lambda_2+\lambda|} \| \tilde{v}_2\|_\infty\right) \|u_1\|_2.
\end{align}
Since $\|u_1\|_2 >0$, we get 
\begin{align*}
     |1+\lambda|\leq  \| \tilde{v}_1\|_\infty + \frac{|\lambda_1\lambda_2-1|}{|\lambda_2+\lambda|} \| \tilde{v}_2\|_\infty.
\end{align*}
But now, given $\mathbf{u} = (u_1, u_2) \in \mathcal{H}$ and $\hat{\mathbf{{u}}}$ the Fourier transform of ${\mathbf{{u}}}$, we have that 
\begin{align*}
    \| u_1\|_\infty + \| u_2\|_\infty \leq \|\hat{{\mathbf{{u}}}}\|_{L^1(\R^2)\times L^1(\R^2)}\leq \sup_{\xi \in \R^2} \|l(\xi)^{-1}\|_2 \|l \hat{{\mathbf{{u}}}}\|_2 = \kappa \|{\mathbf{u}}\|_\mathcal{H}
\end{align*}
where
\[
\kappa \bydef  \sup_{\xi \in \R^2} \|l(\xi)^{-1}\|_2 
\]
 and where we used Parseval's identity for the last step. But now, we define $v_1 \bydef 2u_{0,1}u_{0,2} + 2 u_{0,1} - 3\lambda_1 u_{0,1}$ and $v_2 \bydef u_{0,1}^2$. Using \eqref{eq : u0 for gray scott}, we have that 
 \small{
\begin{align*}
    \|\tilde{v}_1 - v_1\|_\infty &\leq 2 \|u_{0,1}\|_\infty \|u_{0,2}-\tilde{u}_2\|_\infty + 2 \|\tilde{u}_1-u_{0,1}\|_\infty (1+\|\tilde{u}_2\|_\infty) + 3 \lambda_1 \|\tilde{u}_1 - u_{0,1}\|_\infty\|\tilde{u}_1 + u_{0,1}\|_\infty \\
    &\leq \alpha_1\|\tilde{u}_1-u_{0,1}\|_\infty + \alpha_2\|\tilde{u}_2-u_{0,2}\|_\infty
\end{align*}
}\normalsize
where 
\begin{align*}
    \alpha_1 \bydef 2(1+ \|u_{0,2}\|_\infty + \kappa r_0) + 3 \lambda_1(2\|u_{0,1}\|_\infty+ \kappa r_0) ~~ \text{ and } ~~ \alpha_2 \bydef 2\|u_{0,1}\|_\infty. 
\end{align*}
Similarly, we get
\begin{align*}
    \|\tilde{v}_2 - v_2\|_\infty \leq (\|u_{0,1}\|_\infty + \kappa r_0)\|\tilde{u}_1-u_{0,1}\|_\infty.
\end{align*}
Combining the above, we have 
\begin{align*}
    &\| \tilde{v}_1 - v_1\|_\infty + \frac{|\lambda_1\lambda_2-1|}{|\lambda_2+\lambda|} \| \tilde{v}_2 - v_2\|_\infty\\
    \leq & ~\alpha_1\|\tilde{u}_1-u_{0,1}\|_\infty + \alpha_2\|\tilde{u}_2-u_{0,2}\|_\infty + \frac{|\lambda_1\lambda_2-1|}{|\lambda_2+\lambda|} (\|u_{0,1}\|_\infty + \kappa r_0)\|\tilde{u}_1-u_{0,1}\|_\infty\\
    \leq & ~ \max\left\{\alpha_2, ~ \alpha_1 + \frac{|\lambda_1\lambda_2-1|}{|\lambda_2+\lambda|} (\|u_{0,1}\|_\infty + \kappa r_0)\right\} (\|\tilde{u}_1-u_{0,1}\|_\infty + \|\tilde{u}_2-u_{0,2}\|_\infty)\\
    \leq &~ \max\left\{\alpha_2, ~ \alpha_1 + \frac{|\lambda_1\lambda_2-1|}{|\lambda_2+\lambda|} (\|u_{0,1}\|_\infty + \kappa r_0)\right\} \kappa r_0.
\end{align*}
This implies that we can choose $t < -\lambda_{max}$, where $\lambda_{max} >0$ is chosen big enough such that
\begin{align*}
    |1+\lambda| > \| v_1\|_\infty + \frac{|\lambda_1\lambda_2-1|}{|\lambda_2+\lambda|} \|v_2\|_\infty + \max\left\{\alpha_2, ~ \alpha_1 + \frac{|\lambda_1\lambda_2-1|}{|\lambda_2+\lambda|} (\|u_{0,1}\|_\infty + \kappa r_0)\right\} \kappa r_0 
\end{align*}
for all $\lambda > \lambda_{max}$.

Similarly as for the proof of Theorem \ref{th : unstable square solution}, the above reasoning will be useful when computing the bounds $\mathcal{C}_1$ and $\mathcal{C}_2$. In particular, the computation of the bounds of Lemma \ref{lem : link of spectrum unbounded to U0} and \ref{lem : link of spectrum Fourier coeff in q} follows the steps of the proof of Theorem \ref{th : unstable square solution} and can be computed explicitly thanks to the analysis of Section 4 in \cite{cadiot20242dgrayscottequationsconstructive}. We implement such bounds in \cite{julia_cadiot} and, applying Theorem \ref{th : gershgorin unbounded}, we get the following.

We obtain that $D\mathbb{F}_{cc}(\tmu)$ possesses an eigenvalue with positive real part $\nu_1$ such that $Re(\nu_1) \in [1.056, 1.061]$ and the rest of its spectrum has a negative real part. $D\mathbb{F}_{cs}(\tmu)$  possesses an eigenvalue  $\nu_2$ such that $\nu_2 \in \{z \in \mathbb{C},~ |Re(z)| \leq 0.0036 \text{ and } |Im(z)| \leq 0.0036\}$ and the rest of its spectrum has a negative real part.  $D\mathbb{F}_{sc}(\tmu)$ has the same spectrum as  $D\mathbb{F}_{cs}(\tmu)$ since $\tmu$ is $D_4$-symmetric. In particular, it has an eigenvalue $\nu_3 \in \{z \in \mathbb{C}, ~ |Re(z)| \leq 0.0036 \text{ and } |Im(z)| \leq 0.0036\}$. Finally, $D\mathbb{F}_{ss}(\tmu)$ only contains eigenvalues with a negative real part. 

First, we proved that $D\mathbb{F}_{ss}(\tmu)$ only contains eigenvalues with a negative real part, which implies that $Ker(D\mathbb{F}_{ss}(\tmu)) = \{0\}$. Using Proposition 2.1 in \cite{cadiot20242dgrayscottequationsconstructive}, we have that $\tmu \in H^\infty(\R^2)\times H^\infty(\R^2)$ and Theorem 6.2 in \cite{cadiot20242dgrayscottequationsconstructive} provides that $\tmu \neq 0$.  We can apply Lemma \ref{lem : radially symmetric} and we obtain that $\tmu$ is radially symmetric. Moreover, using Lemma \ref{lem : radially symmetric} again, we have that the kernel of $D\mathbb{F}(\tmu)$ is at least of dimension 2. This implies that $\nu_2 = \nu_3 =0$. The nonlinear unstability is a direct consequence of \cite{henry_semilinear}.
\end{proof}

\bibliographystyle{abbrv}
\bibliography{biblio}

\end{document}